\documentclass[leqno]{article}
\usepackage{amsmath}
\usepackage{amssymb}
\usepackage{mathtools}
\numberwithin{equation}{section}
\usepackage{amsthm}
\theoremstyle{plain}
\newtheorem{thm}{Theorem}
\newtheorem{lem}[thm]{Lemma}
\newtheorem{dfn}[thm]{Definition}
\newtheorem{rem}[thm]{Remark}
\newtheorem{fact}[thm]{Fact}
\newtheorem{exa}[thm]{Example}
\usepackage[pdftex]{graphicx}
\begin{document}

\title{Ricci flow of discrete surfaces of revolution, and relation to constant Gaussian curvature}
\author{Naoya Suda}
\date{}
\maketitle

\begin{abstract}
Giving explicit parametrizations of discrete constant Gaussian curvature surfaces of revolution that are defined from an integrable systems approach, we study Ricci flow for discrete surfaces, and see how discrete surfaces of revolution have a geometric realization for the Ricci flow that approaches the constant Gaussian curvature surfaces we have 
parametrized.
\end{abstract}

\section{Introduction}
Since Ricci flow is an intrinsic geometric flow, there arises the separate problem of visualizing the Ricci flow as a flow of actual isometric immersions. In particular, for surfaces of revolution in Euclidean $3$-space $\mathbb{R}^3$, there are studies on how Ricci flow can be visualized. Apart from that, there are numerous studies of Ricci flow for discretization using triangulation. Discrete surfaces defined rather via integrable systems are less directly connected with variational problems and geometric flows, but by restricting such discrete surfaces to surfaces of revolution, we are able to analyze Ricci flow. Unlike the smooth case, the Ricci flow which we consider depends on geometric realization, and the flow toward constant Gaussian curvature (we abbreviate this to CGC) discrete surfaces of revolution with singularities such as cones and cuspidal edges will be seen here by numerical calculations. Furthermore, we will describe explicit parametrizations of discrete CGC surfaces of revolution, and using them, we will see relations between discrete and smooth CGC surfaces of revolution.

\section{The Ricci flow for smooth manifolds}
First, we describe what Ricci flow is. Let $M^{n}$ be an $n$-dimensional Riemannian manifold and let $g_{ij}$ and $R_{ij}$ be the metric tensor and Ricci tensor, respectively.
	
	The Ricci flow satisfies
\begin{equation}
\frac{\partial}{\partial t}g_{ij}=-2R_{ij}  .
\end{equation}
This equation is a second order nonlinear partial defferential equation and was introduced by R. S. Hamilton in 1982 (see \cite{Ham}). 

	When $n=2$, we have $R_{ij}=Kg_{ij}$ for the Gaussian curvature $K$. So the Ricci flow becomes 
\begin{equation}
\frac{\partial}{\partial t}g_{ij}=-2Kg_{ij}  . \label{eq:flow1}
\end{equation}
Defining 
\begin{equation}
r:=\frac{\int 2K dV}{\int dV},
\end{equation}
we have the following equation
\begin{equation}
\frac{\partial}{\partial t}g_{ij}=(r-2K)g_{ij}  
\end{equation}
for what is called the normalized Ricci flow. For this flow, the $n$-dimensional volume $\int dV$, in particular the area when $n=2$, remains constant with respect to time $t$. The following result has been proved (see \cite{Cao, Chow}).

\begin{fact}
Let $M^{2}$ be a $2$-dimensional compact Riemannian manifold with Riemannian metric $g_0$. Then, for the normalized Ricci flow, the solution $g(t)$ satisfying $g(0)=g_0$ exists uniquely for all time $t$. Moreover, when we take $t\rightarrow\infty$, $g(t)$ converges to a metric $g(\infty)$ which has constant Gaussian curvature.
\end{fact}

\section{The Ricci flow on smooth surfaces of revolution}

	For the immersion $x_{0}\colon \mathbb{R}^2 \rightarrow \mathbb{R}^3$ defined by
$$x_{0}(u,v)=(f_{0}(u)\cos v, f_{0}(u)\sin v, h_{0}(u)),$$
the induced metric $g_{0}$ is 
\begin{equation*}
g_{0}=((f_{0,u})^2+(h_{0,u})^2)du^2+(f_{0})^2dv^2.
\end{equation*}
For the solution  of Ricci flow $g(t)$ satisfying $g(0)=g_0$, we consider following equations 
\begin{equation}
\left\{ \,
\begin{aligned}
&g(t)=(f_{u}(u,t)^2+h_{u}(u,t)^2)du^2+f(u,t)^2dv^2,\\
&f(u,0)=f_{0}(u),\\
&h(u,0)=h_{0}(u).
\end{aligned}
\right.
\end{equation}
In the range of $(u,t)$ for which smooth functions $f(u,t)>0$, $h(u,t)$ satisfying  the above equations exist, there will be immersions, for each $t$, 
$$x(u,v,t)=(f(u,t)\cos v, f(u,t)\sin v, h(u,t))$$
which move as geometric realizations according to the Ricci flow.

There are results on the existence of surfaces in Euclidean $3$-space whose metrics evolve under Ricci flow. We call such surfaces global extrinsic representations, or geometric realizations, and V. E. Coll, J. Dodd and D. L. Johnson \cite{extr} have shown that a global extrinsic representation exists for the unique Ricci flow initialized by any smoothly immersed surface of revolution that is either compact or non-compact but complete satisfying certain conditions. J. H. Rubinstein and R. Sinclair \cite{visual} have shown that there exists a global extrinsic representation in $\mathbb{R}^{n+1}$ of any Ricci flow that is initialized by a metric $g_0$ such that $(S^{n}, g_0)$ can be isometrically embedded in $\mathbb{R}^{n+1}$ as a hypersurface of revolution, and they created numerical visualizations of Ricci flow. 

\section{Discrete surfaces of revolution}
As we will later consider Ricci flow of discrete surfaces of revolution, in this section we will consider $x\colon \mathbb{Z}^2\rightarrow \mathbb{R}^3$ defined by
\begin{equation}
x(m,n)=\Bigr(f(n)\cos \frac{2\pi m}{l}, f(n)\sin \frac{2\pi m}{l}, h(n)\Bigr), \label{surf}
\end{equation}
where $l$ is the number of divisions in the rotational direction. We want to define  a metric $g$ and a curvature for this discrete surface of revolution $x$. We will first give definitions for more general discrete surfaces $x\colon \mathbb{Z}^2\rightarrow \mathbb{R}^3$, coming from \cite{Bobenko, Bobenko2, Burstall, norypj, Tim}.
For a general discrete surface $x\colon \mathbb{Z}^2\rightarrow \mathbb{R}^3$, not necessarily satisfying \eqref{surf}, we sometimes denote the four vertices $x(m,n)$, $x(m+1,n)$, $x(m+1,n+1)$, $x(m,n+1)$ of a quadilateral by $x_i,x_j,x_k,x_l$, respectively, and we write $(ijkl)$ for this quadrilateral face. We regard every face $(ijkl)$ as a quadrilateral that may sometimes reduce to a triangle if two adjacent vertices become equal.

\begin{dfn}
A discrete surface $x\colon \mathbb{Z}^2\rightarrow \mathbb{R}^3$ is called a circular net if every quadrilateral $(ijkl)$ of $x$ has vertices lying on a circle.  For a circular net, to define the unit normal vector, on a quadrilateral $(ijkl)$ we have the freedom to arbitrarily define a unit normal vector $\nu_i$ at one vertex $i$, and we can repeatedly define $\nu_*$ at any adjacent vertex $*=j, l$ satisfying the following conditions
\begin{enumerate}
\item[$(1)$] $\nu_i$, $x_*-x_i$, $\nu_*$ all lie in one plane,
\item[$(2)$] $|\nu_i|=|\nu_*|=1$,
\item[$(3)$] the angles between $x_*-x_i$ and $\nu_i$, and between $x_i-x_*$ and $\nu_*$ are equal but oppositely oriented.
\end{enumerate}
Repeatedly applying this along edges, we can define a global normal $\nu\colon \mathbb{Z}^2\rightarrow S^2$, where $S^2=\{\vec{x} \in \mathbb{R}^3 \mid | \vec{x} |=1\}$.
\end{dfn}

\begin{rem}
The notion of normal in Definition 3 is well defined, see for example \cite{Bobenko2, Burstall, norypj}.
\end{rem}

\begin{dfn}
For a circular net $x$, we define the discrete partial derivatives by
\begin{equation}
x_u:=\frac{1}{2}(x_k-x_j)+\frac{1}{2}(x_l-x_i),\;x_v:=\frac{1}{2}(x_k-x_l)+\frac{1}{2}(x_j-x_i)
\end{equation}
on each face $(ijkl)$. We similarly define the partial derivatives for its normal vector field $\nu$.
\end{dfn}

\begin{dfn}
For a circular net $x$ and its unit normal vector $\nu$, we define the first fundamental form (the induced metric), second fundamental form and shape operator by
\begin{equation}
\mathrm{I}:=
\begin{pmatrix}
   x_u\cdot x_u & x_u\cdot x_v \\
   x_v\cdot x_u & x_v\cdot x_v
\end{pmatrix}
,\;\;
\mathrm{I\hspace{-1.2pt}I}:=
\begin{pmatrix}
   x_u\cdot \nu_u & x_u\cdot \nu_v \\
   x_v\cdot \nu_u & x_v\cdot \nu_v
\end{pmatrix}
,\;\;\mathit{S}:=\mathrm{I}^{-1}\rm{I\hspace{-1.2pt}I} \label{eq:metric}
\end{equation}
on each face $(ijkl)$. The eigenvalues $(k_1,k_2)$ and eigenvectors of the shape operator \it{S} are the principal curvatures and curvature line fields of the face. Furthermore, we define the Gaussian curvature and the mean curvature by
\begin{equation}
K:=k_1k_2,\;\;H:=\frac{1}{2}(k_1+k_2). \label{eq:curvd}
\end{equation}
\end{dfn}

\noindent For the $\mathrm{I}$ and $\mathit{S}$ in \eqref{eq:metric}, we sometimes write 
$$\mathrm{I}=
\begin{pmatrix}
   g_{11} & g_{12} \\
   g_{21} & g_{22}
\end{pmatrix}
,\:\:\mathit{S}=
\begin{pmatrix}
 s_{11} & s_{12} \\
 s_{21} & s_{22}
\end{pmatrix}.$$

\noindent Next, we will note that the definitions of curvatures are the same as the definitions produced by mixed area and the Steiner formula.

\begin{dfn}
For planar polygons $P=(p_0, \cdots ,p_k), Q=(q_0, \cdots ,q_k)$ with corresponding edges parallel to each other and unit normal $N$ that is perpendicular to the planes of $P$ and $Q$, we define the mixed area by
$$A(P,Q)=\frac{1}{4}\sum_{j=0}^{k}(\mathrm{det}(p_j,q_{j+1},N)+\mathrm{det}(q_j,p_{j+1},N).$$
When $P=Q$, we write $A(P)$ for $A(P,P)$.
\end{dfn}

\noindent For a circular net $x$ and its unit normal vector $\nu$, we can consider the mixed area, denoted by $A(\cdot,\cdot)_{ijkl}$, on each face $(ijkl)$.

\begin{lem}[\cite{Tim}]
For a circular net $x$ and its unit normal vector $\nu$, the mixed area on each face $(ijkl)$ satisfies 
$$A(x)_{ijkl}=\mathrm{det} (x_u,x_v,N),\;\;A(\nu)_{ijkl}=\mathrm{det}(\nu_u,\nu_v,N),$$
$$A(x,\nu)_{ijkl}=\frac{1}{2}(\mathrm{det}(x_u,\nu_v,N)+\mathrm{det}(\nu_u,x_v,N)),$$
where $N$ is the unit normal for the face.
\end{lem}

\begin{lem}[\cite{Tim}]
For a circular net $x$ and its unit normal vector $\nu$, we have
\begin{equation*}
\begin{split}
A(x+t\nu)_{ijkl}&=A(x)_{ijkl}+2tA(x,\nu)_{ijkl}+t^2A(\nu)_{ijkl}\\
&=(1+2tH+t^2K)A(x)_{ijkl}
\end{split}
\end{equation*}
for $t\in\mathbb{R}$ on each face $(ijkl)$, and for $H$ and $K$ as in \eqref{eq:curvd}.
\end{lem}

We now consider the surface as in \eqref{surf}. Since every quadrilateral of a discrete surface of revolution is an isosceles trapezoid, its vertices lie on a circle. Therefore the surface \eqref{surf} is a circular net. Considering the three conditions of Definition 3, and wishing to preserve symmetry, the unit normal vector $\nu$ can be written as
$$\nu(n,m)=\Bigr(a(n)\cos \frac{2\pi m}{l}, a(n)\sin \frac{2\pi m}{l}, b(n)\Bigr),$$
where $a(n)$ and $b(n)$ satisfy
$$\left\{ \,
\begin{aligned}
&a(n+1)=\\
&a(n)-\frac{2(f(n+1)-f(n))(a(n)(f(n+1)-f(n))+b(n)(h(n+1)-h(n))}{(f(n+1)-f(n))^2+(h(n+1)-h(n))^2},\\
&b(n+1)=\\
&b(n)-\frac{2(h(n+1)-h(n))(a(n)(f(n+1)-f(n))+b(n)(h(n+1)-h(n))}{(f(n+1)-f(n))^2+(h(n+1)-h(n))^2}.
\end{aligned}
\right.$$

\begin{rem}
When $f(k)=0$ ($k\in\mathbb{Z}$), the face with  $x(m,k-1),x(m+1,k-1),x(m+1,k),x(m,k)$ as vertices is a triangle. In this case, when defining the unit normal vector $\nu$ as above, even though $x(m,k)$ is the same point for any $m$, $\nu(m,k)$ will depend on $m$. 
\end{rem}

The metric and curvatures are 
$$g_{11}=(f(n+1)-f(n))^2\cos^2 \frac{\pi}{l}+(h(n+1)-h(n))^2,$$
$$g_{12}=g_{21}=0,\;\:g_{22}=(f(n+1)+f(n))^2\sin^2 \frac{\pi}{l},$$
\begin{equation}
K=\frac{a(n+1)^2-a(n)^2}{f(n+1)^2-f(n)^2},\;\;H=\frac{f(n+1)a(n+1)-f(n)a(n)}{f(n+1)^2-f(n)^2}. \label{eq:curv}
\end{equation}
(See also \cite{Bobenko}.) Note that the metric and curvatures are independent of  the rotational parameter $m$. Furthermore by Lemma 7, we have $A(x)_{ijkl}=\sqrt{g_{11}g_{22}}$. 
Then the area $A(x)_{ijkl}(n,m)$ is also independent of $m$, so we may denote it simply as $A(x)(n)$ when all is clear from context.

\section{Comparison of smooth and discrete CGC surfaces of revolution}
In this next section we will define the flow expected to converge to a discrete CGC surface of revolution. In this section, in preparation for that, we describe a parametrization of discrete CGC surfaces and compare them with the smooth CGC surfaces of revolution.

For the smooth immersion defined by
$$x(u,v)=(f(u)\cos v, f(u)\sin v, h(u)),$$
the unit normal vector can be written as
$$\nu(u,v)=(a(u)\cos v, a(u)\sin v, b(u)),$$
where $a(u)=\frac{h_{u}}{\sqrt{f_{u}^2+h_{u}^2}}$ and $b(u)=\frac{-f_{u}}{\sqrt{f_{u}^2+h_{u}^2}}$. 
Then the Gaussian curvature is 
$$K=\frac{(a^2)_{u}}{(f^2)_{u}}.$$
When $K$ is constant, we have
$$K=\frac{a(\tilde{u})^2-a(u)^2}{f(\tilde{u})^2-f(u)^2}$$
for all $u,\tilde{u}$ with $\tilde{u}>u$.
Therefore, setting $\tilde{u}=u_{n+1}$ and $u=u_{n}$, we have 
$$K=\frac{a(u_{n+1})^2-a(u_{n})^2}{f(u_{n+1})^2-f(u_{n})^2}$$
for the functions $a(u)$ and $f(u)$ used in the smooth case. This is perfectly analogous to the equation \eqref{eq:curv} for $K$ in the discrete case, where \eqref{eq:curv} uses the functions $a(n)$ and $f(n)$ given for the discrete case. We conclude that the functions $a(\cdot)$ and $f(\cdot)$ can be taken to be the same in both the smooth and discrete cases, with appropriate choices of the $u_{n}$. 

In the smooth case, CGC $K=1$ surfaces of revolution can be parameterized using
$$(f(u),h(u))=\Bigr(p\cos u,\int_{0}^{u} \sqrt{1-p^2\sin^2 s}\;ds\Bigr)$$
for $p$ a real parameter, and the unit normal vector can be parameterized by
$$(a(u),b(u))=(\sqrt{1-p^2\sin^2 u},p\sin u)$$
 (see \cite{Umehara}). As noted above, we can parametrize discrete CGC $K=1$ surfaces of revolution using these $f(u)$ and $a(u)$, because $b(u)$ and $h(u)$ can be determined from $f(u)$ and $a(u)$. This is clearly so for $b(u)$ (up to sign), and we now explain how to determine $h(u)$ as well. We now rewrite $f(n)$, $a(n)$, $b(n)$ as $f(u_n)$, $a(u_n)$, $b(u_n)$, respectively. From the definition of unit normal vectors of circular nets, we have
$$\frac{h(n+1)-h(n)}{f(n+1)-f(n)}=\frac{b(n+1)-b(n)}{a(n+1)-a(n)},$$
and so
\begin{equation}
h(n)=h(0)+\sum_{i=1}^{n} \frac{b(i)-b(i-1)}{a(i)-a(i-1)}(f(i)-f(i-1)) \label{eq:h}
\end{equation}
is also determined. Thus, we can prove the following result.

\begin{thm}
For a discrete CGC $K=1$ surface of revolution as in \eqref{surf}, we assume that 
$$f(n)\geq 0$$
for all $k_1\leq n\leq k_2$ for some integers $k_1\leq0$ and $k_2\geq0$. If we have 
$$f(0)=p,\;h(0)=0,\;a(0)=1,\;b(0)=0,$$
$$a(n)=\sqrt{1-p^2\sin^2 u_n},\;b(n)=p\sin u_n$$
for some $p>0$ and some sequence $-\pi/2\leq u_{k_1}<\cdots<u_{-1}<u_0<u_1<\cdots<u_{k_2}\leq\pi/2$, then the profile curve of the surface is given by
\begin{multline}
(f(n),h(n))=\\
\Bigr(p\cos u_{n},\sum_{i=1}^{n} \frac{(p\sin u_{i}-p\sin u_{i-1})(p\cos u_{i}-p\cos u_{i-1})}{\sqrt{1-p^2\sin^2 u_{i}}-\sqrt{1-p^2\sin^2 u_{i-1}}}\Bigr)\;\; (0<n\leq k_2),
\end{multline}
\begin{multline}
(f(n),h(n))=\\
\Bigr(p\cos u_{n},\sum_{i=1}^{-n} \frac{(p\sin u_{-i}-p\sin u_{-i+1})(p\cos u_{-i}-p\cos u_{-i+1})}{\sqrt{1-p^2\sin^2 u_{-i}}-\sqrt{1-p^2\sin^2 u_{-i+1}}}\Bigr)\;\; (k_1\leq n<0).
\end{multline}
\end{thm}
\begin{proof}
When $n=1$, from $f(1)\geq0$ and the difference equation
$$1=\frac{a(1)^2-a(0)^2}{f(1)^2-f(0)^2},$$
we have $f(1)=p\cos u_1$.
Inductively, on $0\leq n\leq k_2$, we find that $f(n)=p\cos u_n$.
Considering in the same way, we also have $f(n)=p\cos u_n$ for $k_1\leq n\leq0$. Then the function $h$ is determined by \eqref{eq:h}.
\end{proof}

\begin{rem}
More generally, for Gaussian curvature $K=c$ for $c>0$, we have the functions
\begin{multline}
(f(n),h(n))=\Bigr(p\cos(\sqrt{c}u_{n}),\\
\sum_{i=1}^{n} \frac{(p\sqrt{c}\sin(\sqrt{c}u_{i})-p\sqrt{c}\sin(\sqrt{c}u_{i-1}))(p\cos(\sqrt{c}u_{i})-p\cos(\sqrt{c}u_{i-1}))}{\sqrt{1-p^2c\sin^2(\sqrt{c}u_{i})}-\sqrt{1-p^2c\sin^2(\sqrt{c}u_{i-1})}}\Bigr), \label{eq:CGC}
\end{multline}
\begin{equation}
(a(n),b(n))=\bigr(\sqrt{1-p^2c\sin^2(\sqrt{c}u_{n})},p\sqrt{c}\sin(\sqrt{c}u_{n})\bigr)\label{eq:gaCGC1}
\end{equation}
which can parameterize CGC $K=c$ surfaces of revolution and their normal vectors. Figure \ref{fig:CGC1} shows smooth and discrete CGC $K=1$ surfaces of revolution.
\end{rem}

\begin{figure}[htbp]
\hspace{-0.3cm}
 \centering
\begin{minipage}[t]{0.2\textwidth}
 \includegraphics[width=1.3cm]{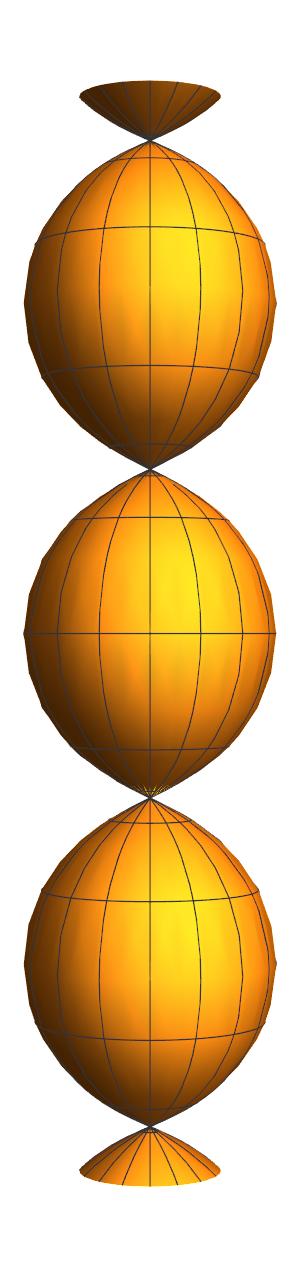}
\end{minipage}
\begin{minipage}[t]{0.2\textwidth}
 \includegraphics[width=2.0cm]{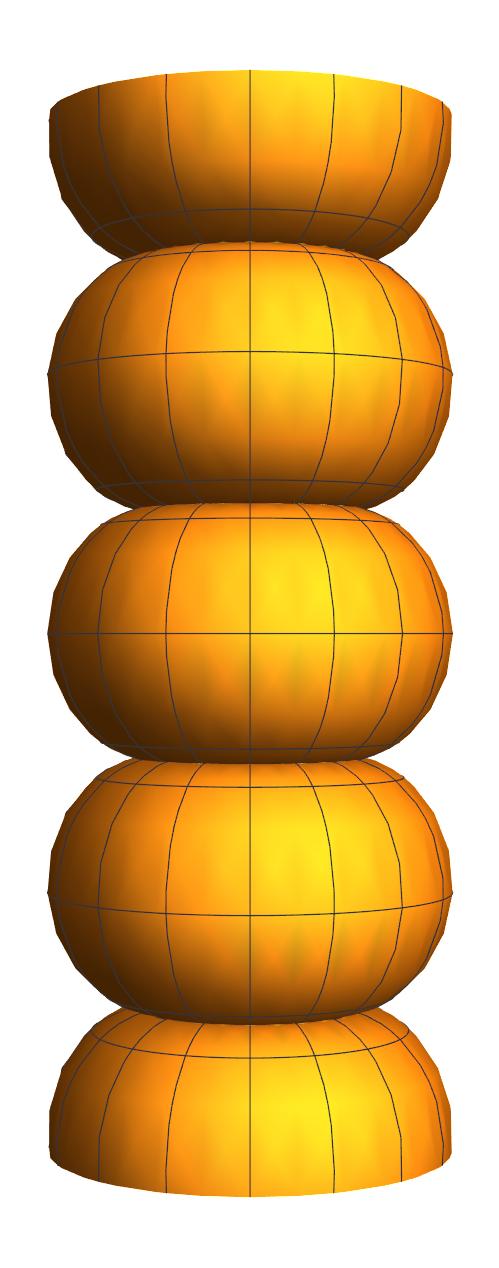}
\end{minipage}
\begin{minipage}[t]{0.23\textwidth}
 \includegraphics[width=2.7cm,]{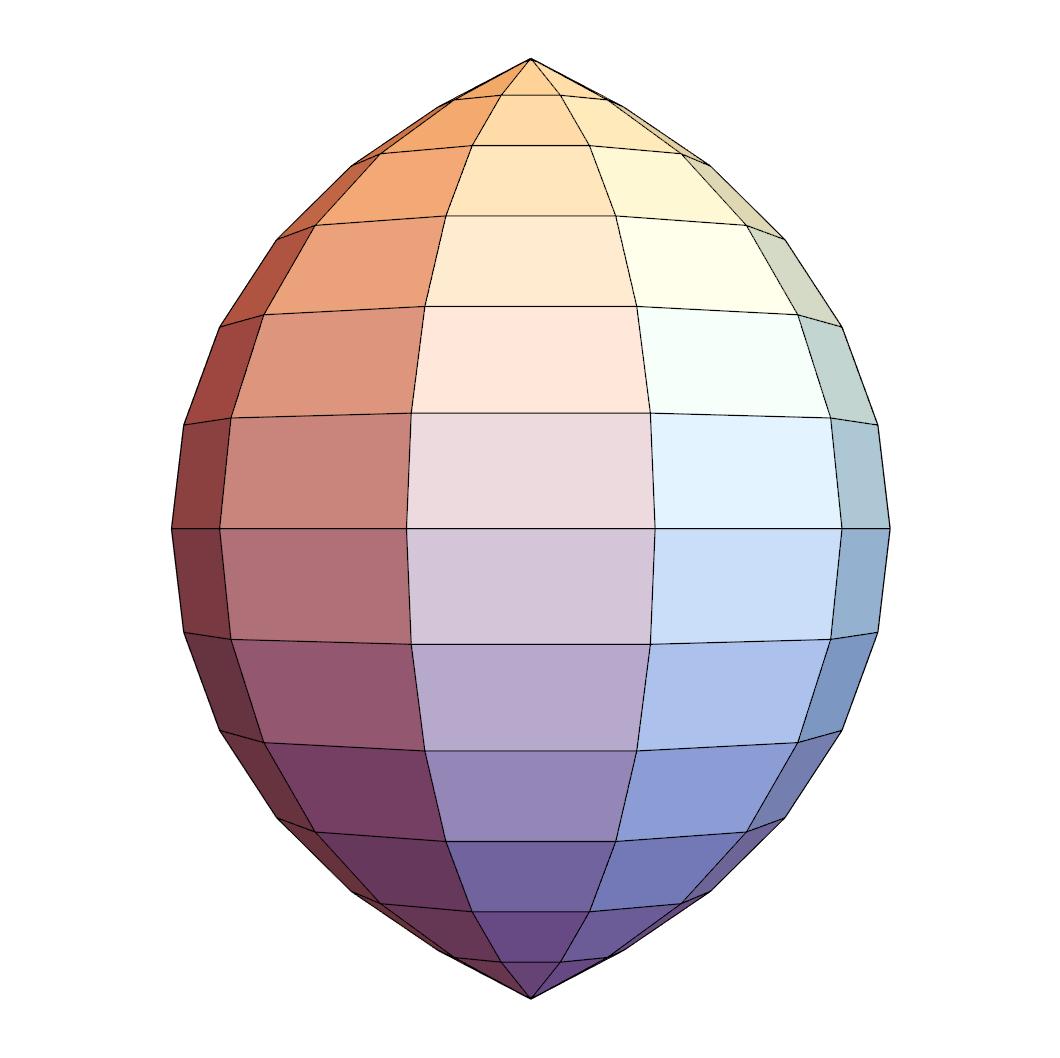}
\end{minipage}
\hspace{0.3cm}
\begin{minipage}[t]{0.23\textwidth}
 \includegraphics[width=2.7cm]{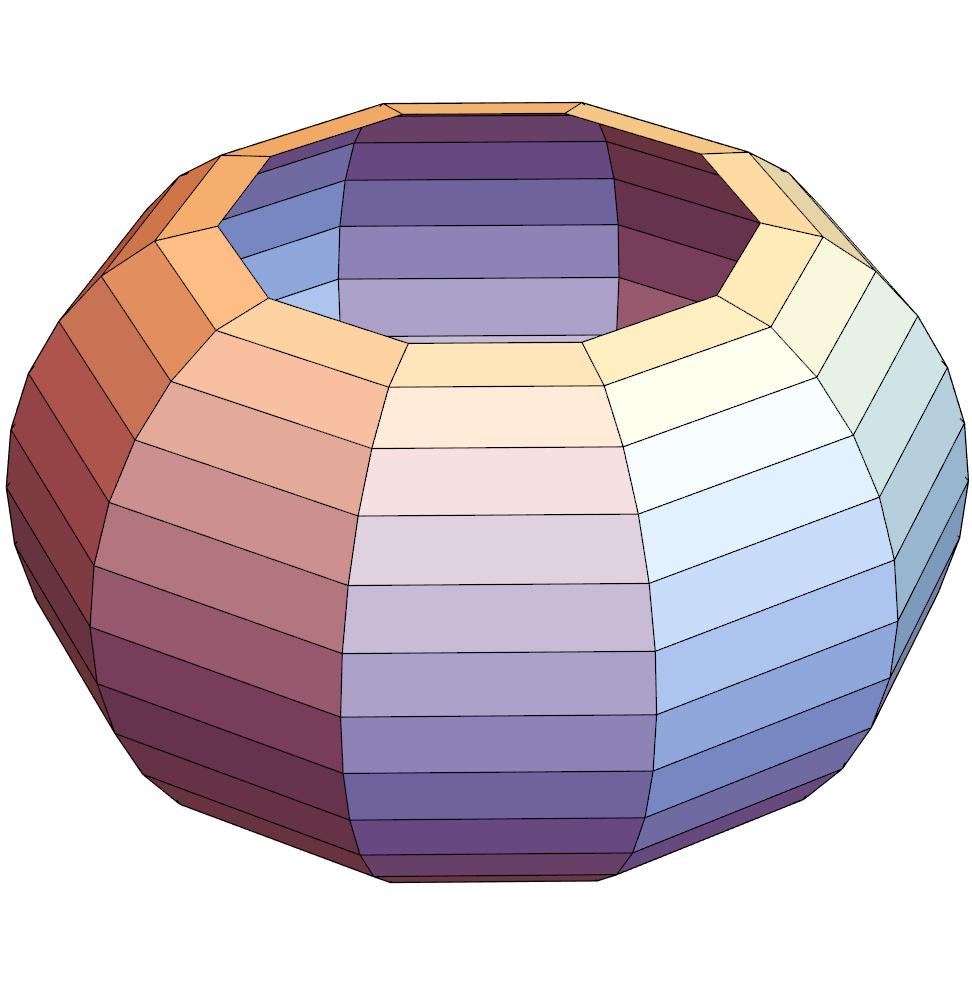}
\end{minipage}
\caption{Left: Smooth CGC $K=1$ surface for $p=0.9$. Second from left: Smooth CGC $K=1$ surface for $p=1.2$. Second from right: Discrete CGC $K=1$ surface for $p=0.9$ and $u_n=\frac{\pi}{12}n$. Right: Discrete CGC $K=1$ surface for $p=1.2$ and $u_n=\arcsin(\frac{1}{1.2}) \frac{n}{6}$.} 
\label{fig:CGC1}
\end{figure}

The case for the Gaussian curvature $K=-1$ can be described explicitly as well. For the smooth case, CGC $K=-1$ surfaces of revolution can be parametrized using any of the following three equations:
\begin{equation}
(f(u),h(u))=\Bigr(\frac{1}{\cosh u}, u-\tanh u\Bigr), \label{eq:pseudo}
\end{equation}
\begin{equation}
(f(u),h(u))=\Bigr(p\cosh u, \int_{0}^{u} \sqrt{1-p^2\sinh^2 s}\;ds\Bigr), \label{eq:cosh}
\end{equation}
\begin{equation}
(f(u),h(u))=\Bigr(q\sinh u, \int_{0}^{u} \sqrt{1-q^2\cosh^2{s}}\;ds\Bigr) \label{eq:sinh}
\end{equation}
for some $p>0$ and some $0<q<1$ (see \cite{Umehara}). Thus we can prove the following result like as for Theorem $10$.

\begin{thm}
For a discrete CGC $K=-1$ surface of revolution as in \eqref{surf}, we assume that 
$$f(n)\geq 0$$
for all $0\leq n\leq k_2$ for some integer $k_2\geq0$. We consider three cases corresponding to \eqref{eq:pseudo}, \eqref{eq:cosh} and \eqref{eq:sinh}$:$
\begin{enumerate}
\item[$(1)$] If we have 
$$f(0)=1,\;h(0)=0,\;a(0)=0,\;b(0)=1,$$
$$a(n)=\tanh u_n,\;b(n)=\frac{1}{\cosh u_n}$$
for some sequence $u_0<u_1<\cdots<u_{k_2}$, then the profile curve of the surface is given by
\begin{multline}
(f(n),h(n))=\Bigr(\frac{1}{\cosh u_{n}}, \sum_{i=1}^{n} \frac{(\cosh u_{i}-\cosh u_{i-1})^{2}}{\sinh (u_{i}-u_{i-1})\cosh u_{i}\cosh u_{i-1}}\Bigr)\\
(0<n\leq k_2).
\end{multline} 

\item[$(2)$] If we have 
$$f(0)=p,\;h(0)=0,\;a(0)=1,\;b(0)=0,$$
\begin{equation}
a(n)=\sqrt{1-p^2\sinh^2 u_{n}},\;b(n)=-p\sinh u_{n} \label{eq:gaCGC-1c}
\end{equation}
for some $p>0$ and some sequence $u_0<u_1<\cdots<u_{k_2}$, then the profile curve of the surface is given by
\begin{multline}
(f(n),h(n))=\\
\Bigr(p\cosh u_{n},\sum_{i=1}^{n} \frac{(-p\sinh u_{i}+p\sinh u_{i-1})(p\cosh u_{i}-p\cosh u_{i-1})}{\sqrt{1-p^2\sinh^2 u_{i}}-\sqrt{1-p^2\sinh^2 u_{i-1}}}\Bigr)\\
 (0<n\leq k_2).
\end{multline}

\item[$(3)$] If we have 
$$f(0)=q\sinh u_0,\;h(0)=0,\;a(0)=0,\;b(0)=1,$$
\begin{equation}
a(n)=\sqrt{1-q^2\cosh^2 u_{n}},\;b(n)=q\cosh u_{n} \label{eq:gaCGC-1h}
\end{equation}
for some $0<q<1$ and some sequence $u_0>u_1>\cdots>u_{k_2}$, then the profile curve of the surface is given by
\begin{multline}
(f(n),h(n))=\\
\Bigr(q\sinh u_{n},\sum_{i=1}^{n} \frac{(q\cosh u_{i}-q\cosh u_{i-1})(q\sinh u_{i}-q\sinh u_{i-1})}{\sqrt{1-q^2\cosh^2 u_{i}}-\sqrt{1-q^2\cosh^2 u_{i-1}}}\Bigr)\\
 (0<n\leq k_2).
\end{multline}
\end{enumerate}
\end{thm}

\begin{rem}
In Theorem $11$ we can extend to $k_1\leq n\leq k_2$ for some integer $k_1\leq 0$, like we did in Theorem $10$. The arguments are completely analogous.
\end{rem}

Figure \ref{fig:CGC-1} shows smooth and discrete CGC $K=-1$ surfaces of revolution. Figure \ref{fig:com1} shows the comparison of the profile curves of the smooth and discrete CGC surfaces of revolution.

\begin{figure}[htbp]
 \hspace{-0.7cm}
 \centering
\begin{minipage}[t]{0.32\textwidth}
 \includegraphics[width=3.0cm]{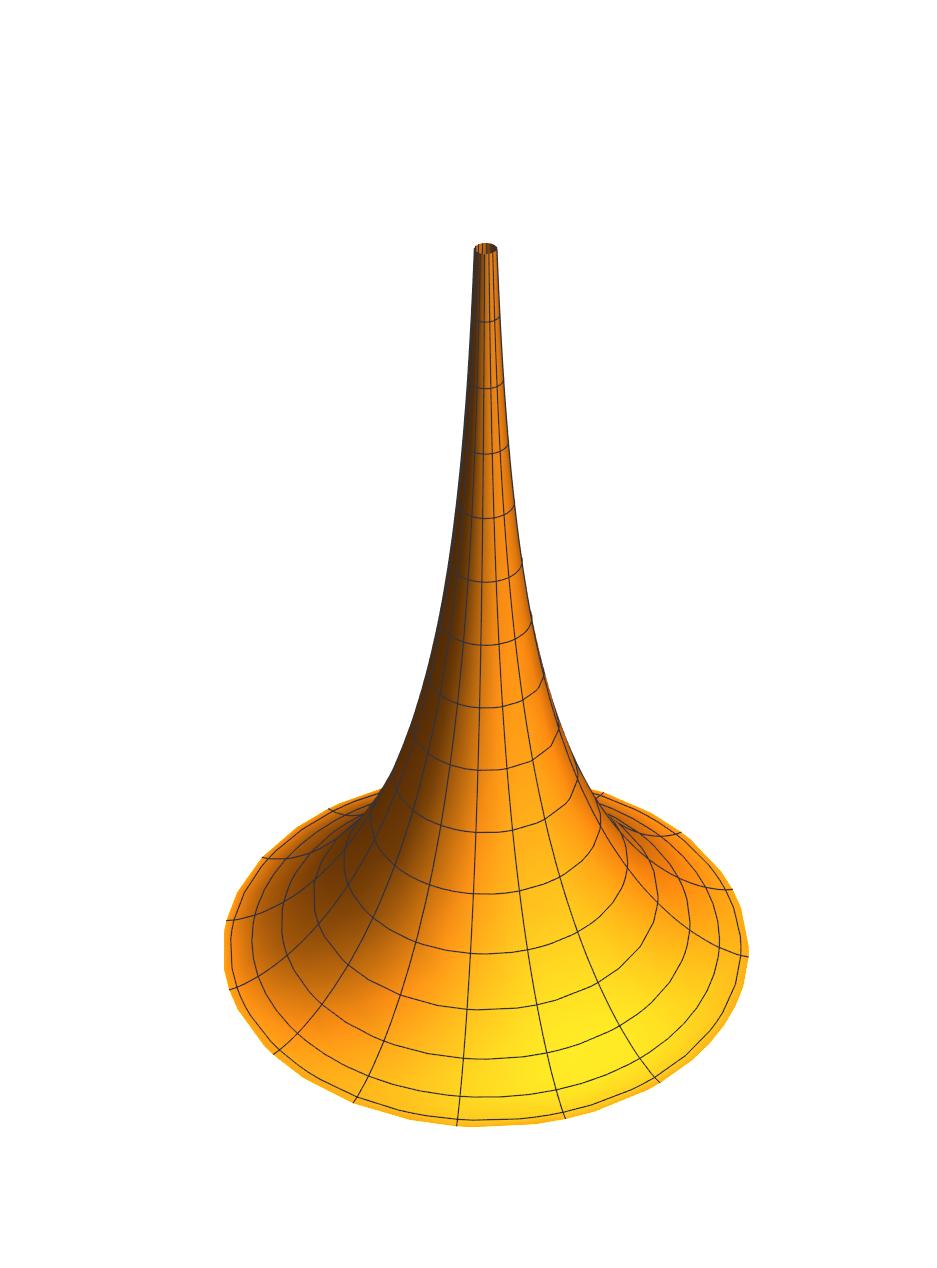}
\end{minipage}
\begin{minipage}[t]{0.32\textwidth}
 \includegraphics[width=3.0cm]{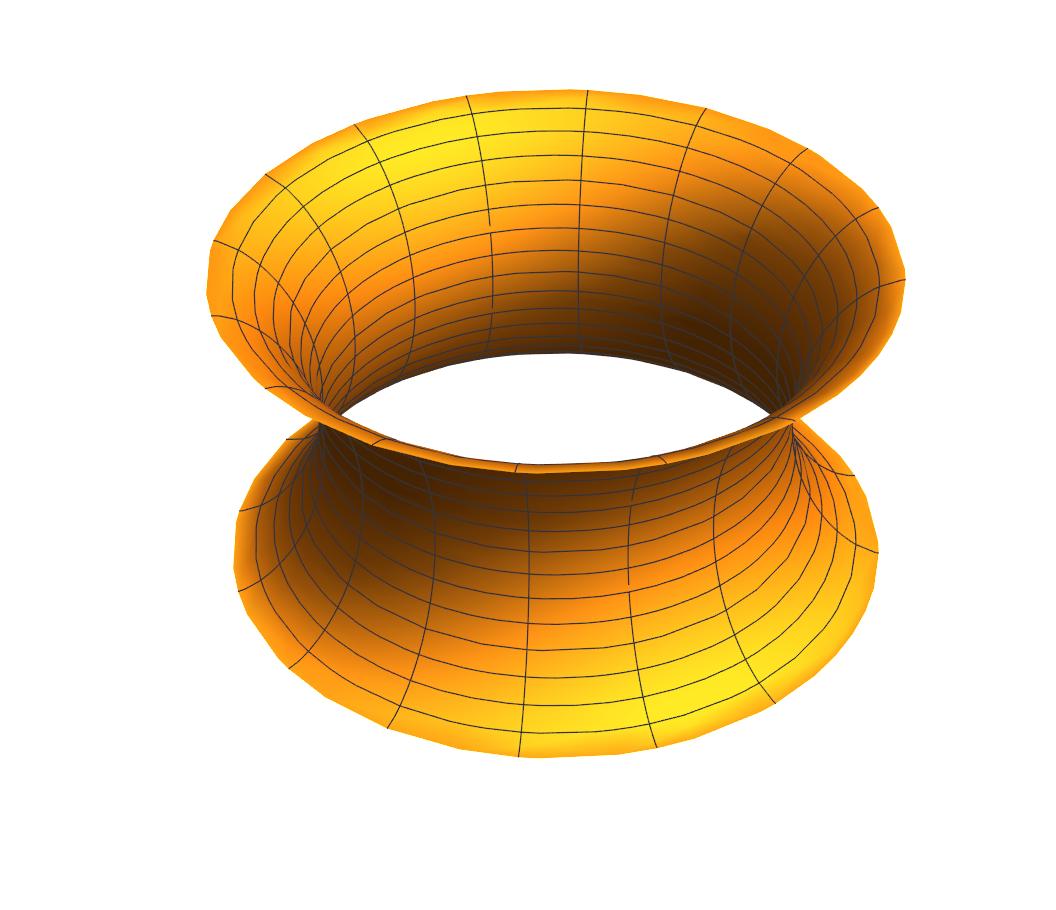}
\end{minipage}
\begin{minipage}[t]{0.32\textwidth}
 \includegraphics[width=3.3cm]{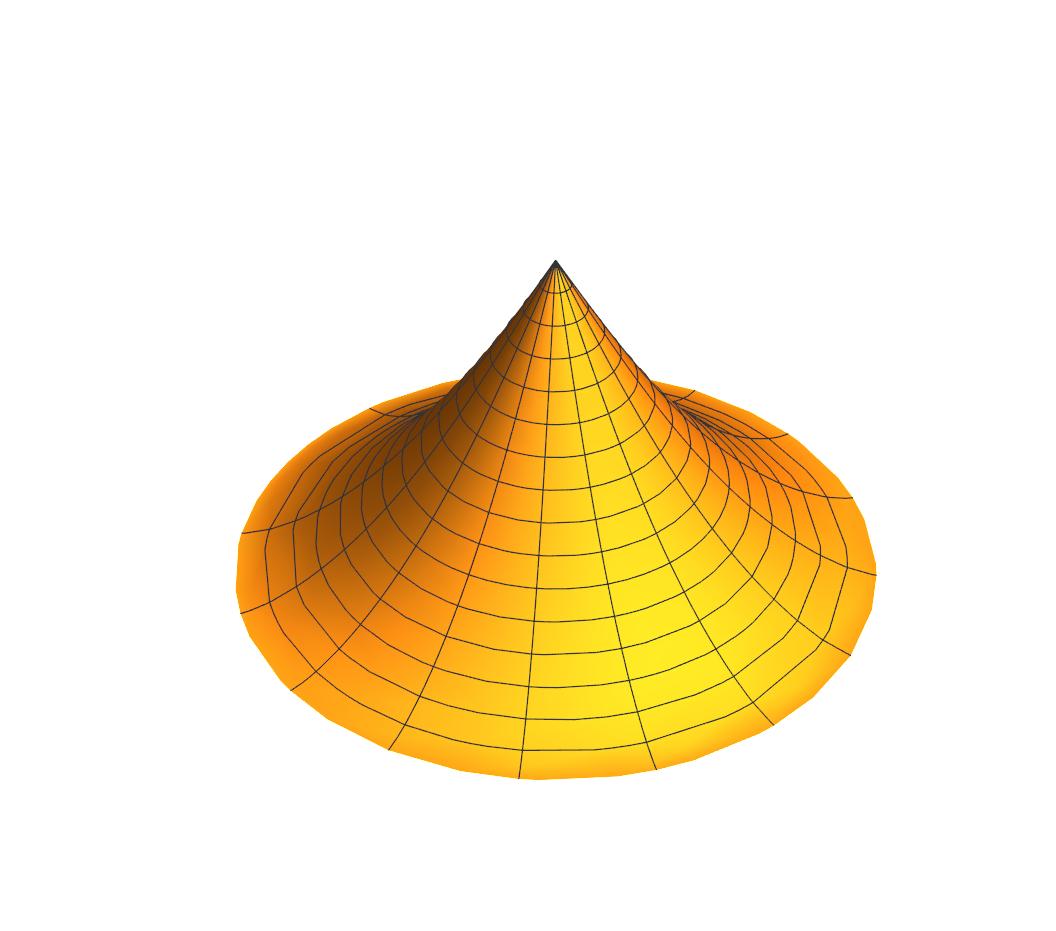}
\end{minipage}
\begin{minipage}[t]{0.322\textwidth} \vspace{-0.7cm} 
 \includegraphics[width=2.7cm]{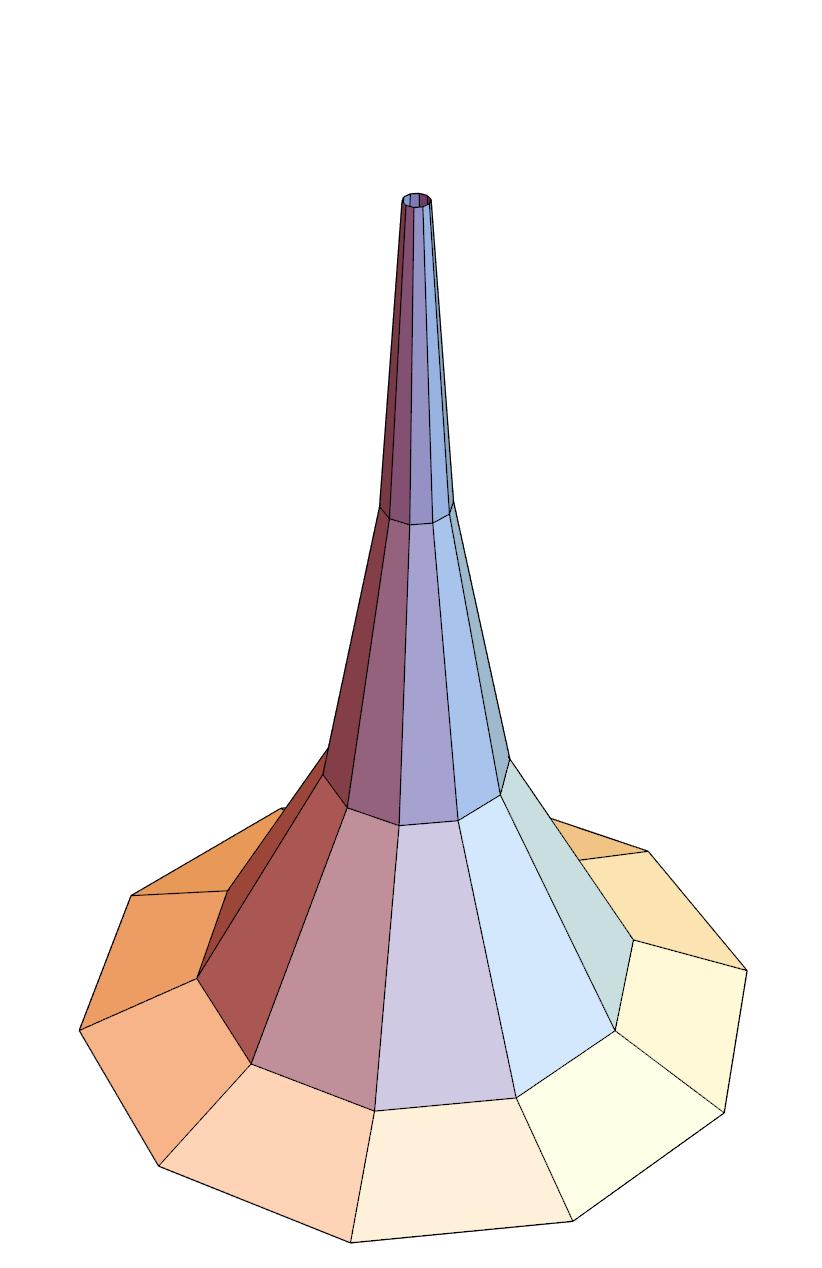}
\end{minipage}
\begin{minipage}[t]{0.322\textwidth} \vspace{0.5cm}
 \includegraphics[width=2.7cm]{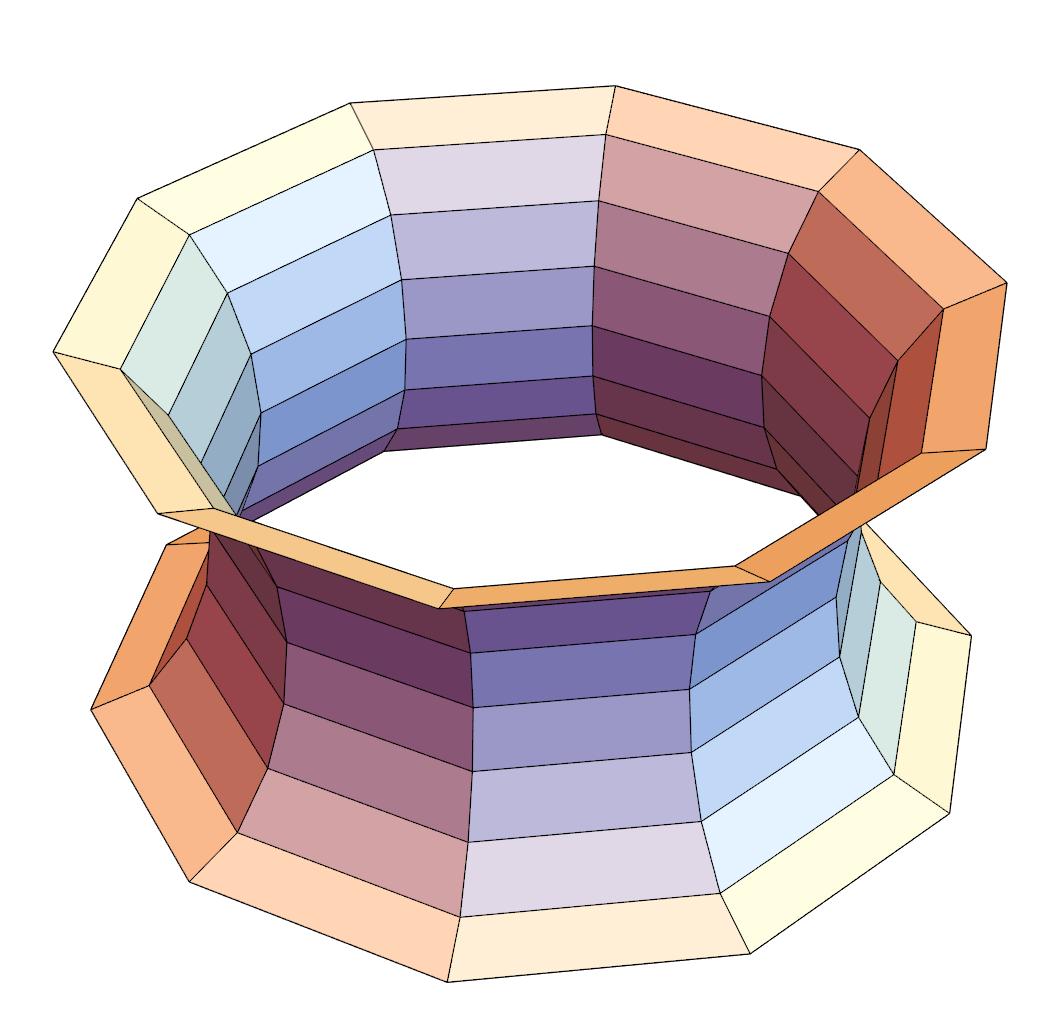}
\end{minipage}
\begin{minipage}[t]{0.322\textwidth} \vspace{0.5cm}
 \includegraphics[width=2.7cm]{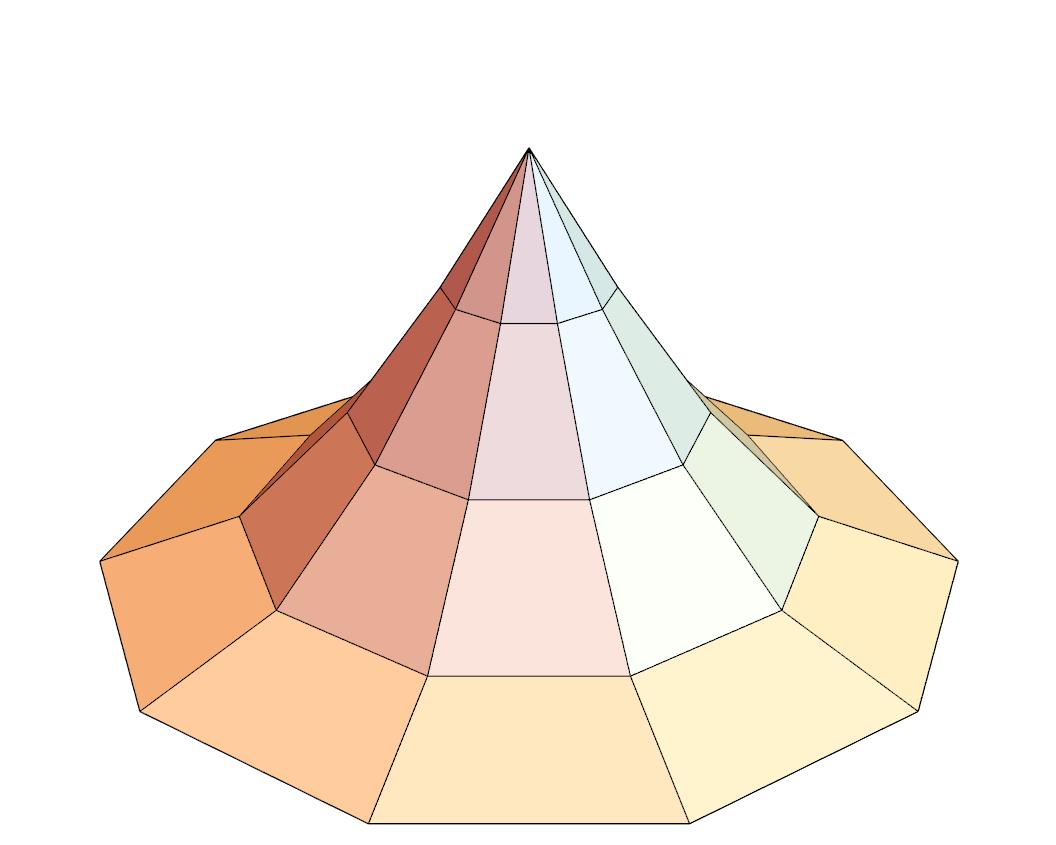}
\end{minipage}
\caption{Upper left: Smooth pseudosphere parameterized by \eqref{eq:pseudo}. Upper middle: Smooth CGC $K=-1$ surface parameterized by \eqref{eq:cosh} for $p=1$. Upper right: Smooth CGC $K=-1$ surface parameterized by \eqref{eq:sinh} for $q=\frac{1}{2}$. Lower left: Discrete pseudosphere for $u_n=n$. Lower middle: Discrete CGC $K=-1$ surface for $p=1$ and $u_n=\frac{\log({1+\sqrt{2}})}{4}n$. Lower right:  Discrete CGC $K=-1$ surface for $q=\frac{1}{2}$ and $u_n=(1-\frac{n}{4})\mathrm{arccosh}2$.} 
\label{fig:CGC-1}
\end{figure}

\begin{figure}[htbp]
 \centering
 \hspace{1.5cm}
\begin{minipage}[t]{0.45\textwidth}
 \includegraphics[width=2.0cm]{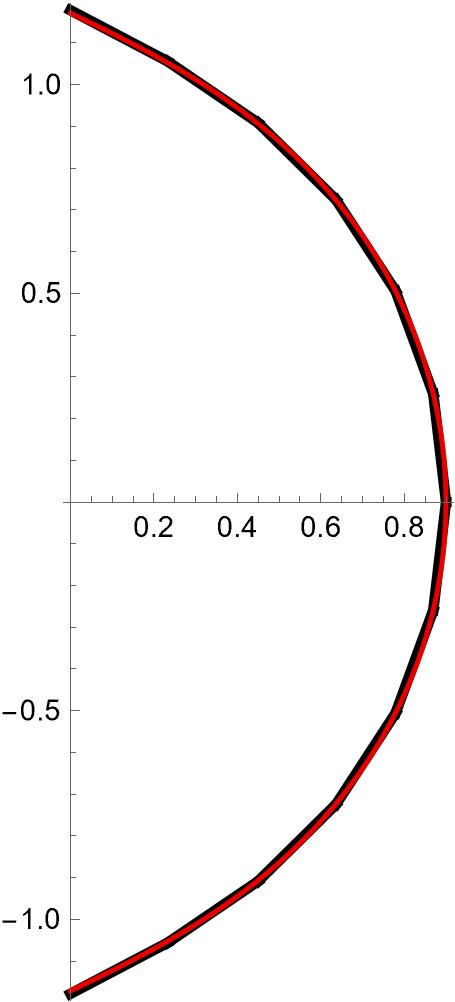}
\end{minipage}
\begin{minipage}[t]{0.4\textwidth}
 \includegraphics[width=1.3cm]{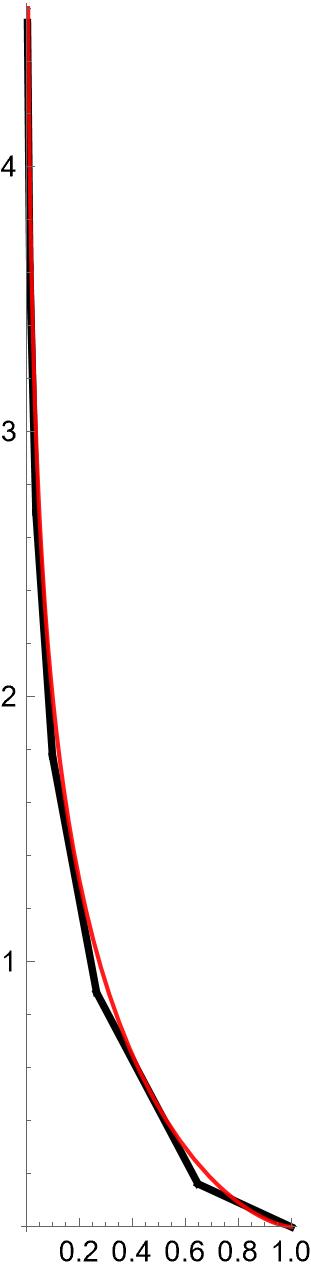}
\end{minipage}
\caption{Left: The profile curves of the spindle type CGC $K=1$ discrete surface and CGC $K=1$ smooth surface in Figure \ref{fig:CGC1}. Right: The profile curves of the discrete pseudosphere and smooth pseudosphere in Figure \ref{fig:CGC-1}.} 
\label{fig:com1}
\end{figure}

\begin{rem}
We can also find that $a(u)$ and $f(u)$ in the case of smooth constant mean curvature surfaces of revolution satisfy the difference equation 
$$H=\frac{f(u_n)a(u_{n})-f(u_{n+1})a(u_{n+1})}{f(u_{n+1})^2-f(u_{n})^2},$$
similarly to an analogous equation for discrete constant mean curvature surfaces of revolution. Therefore, in the same way as for CGC surfaces, we have the functions
$$(f(n),h(n))=\Bigr(\cosh u_{n},\sum_{i=1}^{n} \sinh (u_{i}-u_{i-1})\Bigr)$$
which can parameterize the discrete catenoid. This surface can be obtained by substituting the discrete holomorphic function $g(n,m)=\exp(-u_{n}+\sqrt{-1}\frac{2\pi m}{l})$ into the Weierstrass representation for discrete minimal surfaces.
Figure \ref{fig:cate} shows two discrete catenoid with different $u_n$.
\end{rem}

\begin{figure}[htbp]
 \centering
 \hspace{1.0cm}
\begin{minipage}[t]{0.45\textwidth}
 \includegraphics[width=3cm]{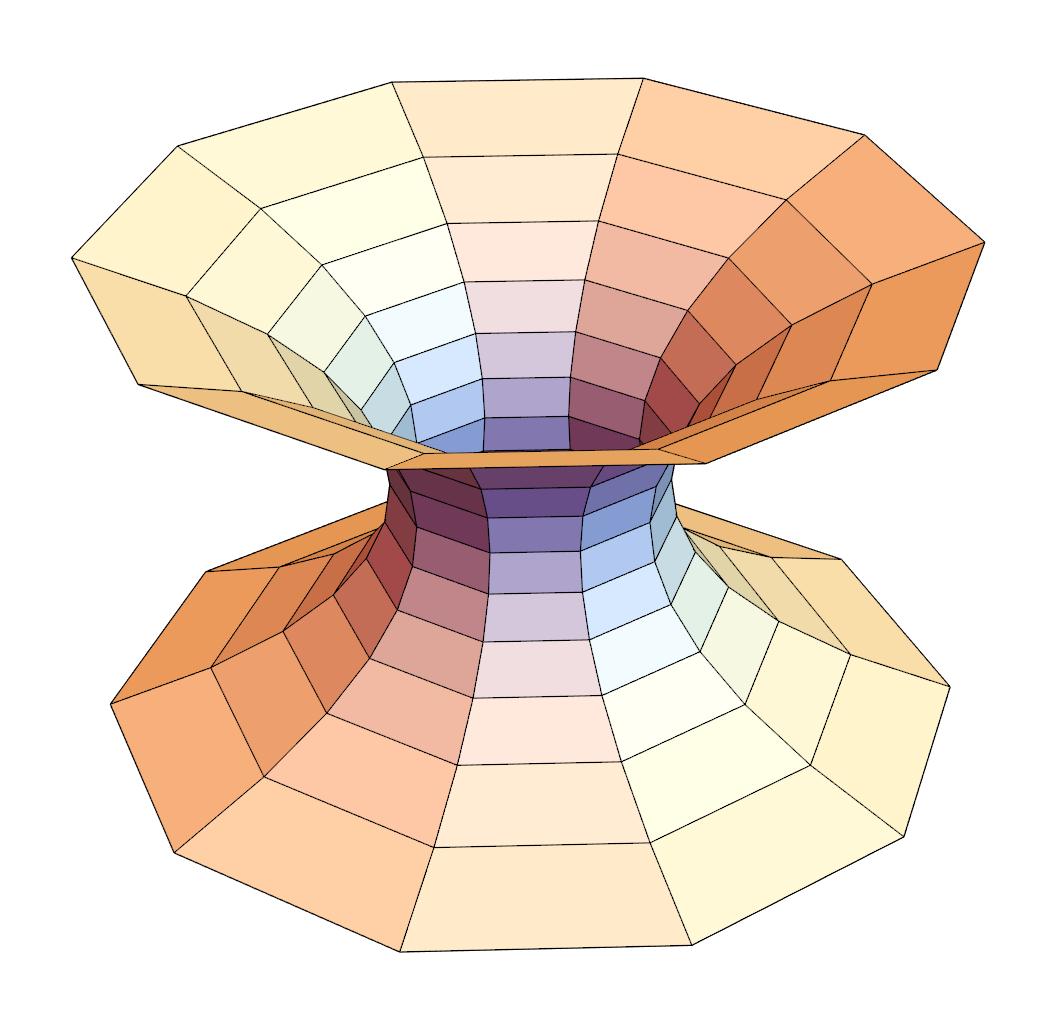}
\end{minipage}
\begin{minipage}[t]{0.4\textwidth}
 \includegraphics[width=3cm]{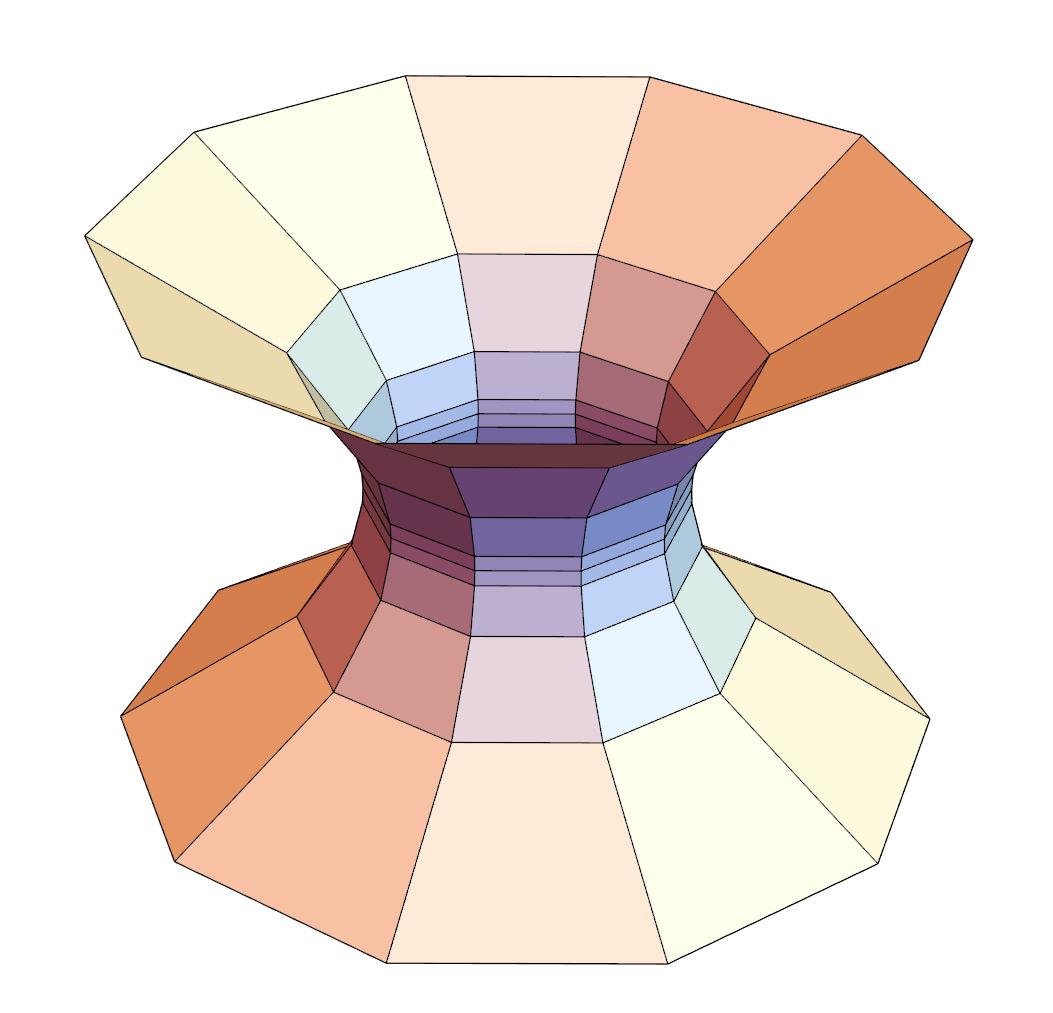}
\end{minipage}
\caption{Left: Discrete catenoid for $u_{n}=0.3n$. Right: Discrete catenoid for $u_{n}=0.1n^2$.} 
\label{fig:cate}
\end{figure}

\begin{rem}
In fact, we also now have explicit parameterizations of discrete constant mean curvature surfaces (Delaunay surfaces) by looking at two particular parallel surfaces of CGC $K$ surfaces with $K=c$ positive, as follows:
\begin{multline*}
(f(n),h(n))=\Bigr(\frac{\epsilon}{\sqrt{c}}\sqrt{1-p^2c\sin^2(\sqrt{c}u_{n})}+p\cos(\sqrt{c}u_{n}),\frac{\epsilon}{\sqrt{c}}p\sqrt{c}\sin(\sqrt{c}u_{n})+\\
\sum_{i=1}^{n} \frac{(p\sqrt{c}\sin(\sqrt{c}u_{i})-p\sqrt{c}\sin(\sqrt{c}u_{i-1}))(p\cos(\sqrt{c}u_{i})-p\cos(\sqrt{c}u_{i-1}))}{\sqrt{1-p^2c\sin^2(\sqrt{c}u_{i})}-\sqrt{1-p^2c\sin^2(\sqrt{c}u_{i-1})}}\Bigr)
\end{multline*}
for $\epsilon=\pm1$. Figure \ref{fig:dela} shows discrete CMC surfaces. In \cite{Inoguchi} there are functions which parameterize smooth Delaunay surfaces. Alternately, using those functions, we can also determine parametrizations of discrete Delaunay surfaces. 
\end{rem}
\vspace{-0.3cm}

\begin{figure}[htbp]
 \hspace{1.3cm}
 \centering
\begin{minipage}[t]{0.45\textwidth}
 \includegraphics[trim= 2.5cm 2.7cm 2.7cm 2.7cm, clip, width=2.3cm]{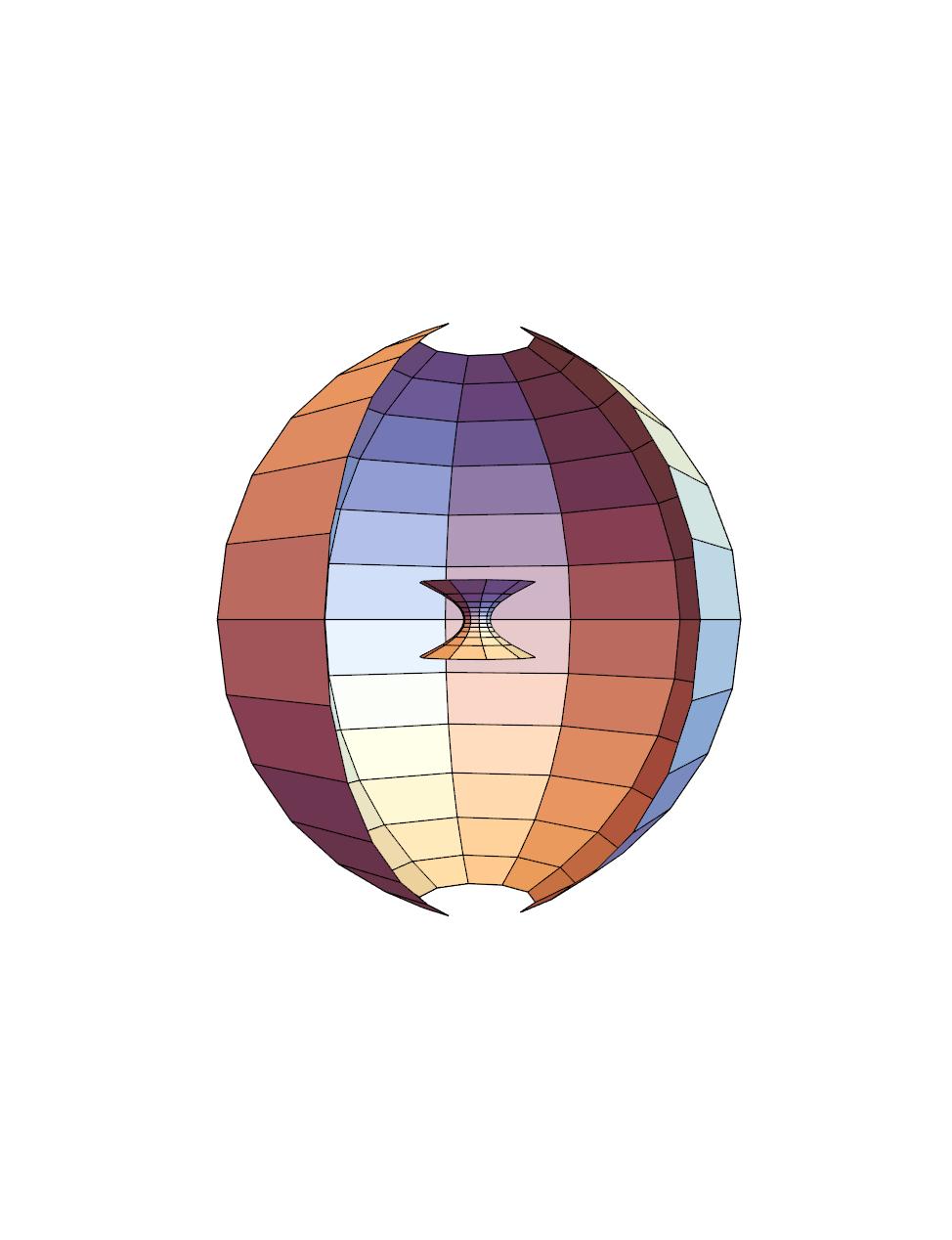}
\end{minipage}
\vspace{-0.5cm}
\begin{minipage}[t]{0.4\textwidth}
 \includegraphics[trim= 2.0cm 2.7cm 2.0cm 2.7cm, clip, width=2.5cm]{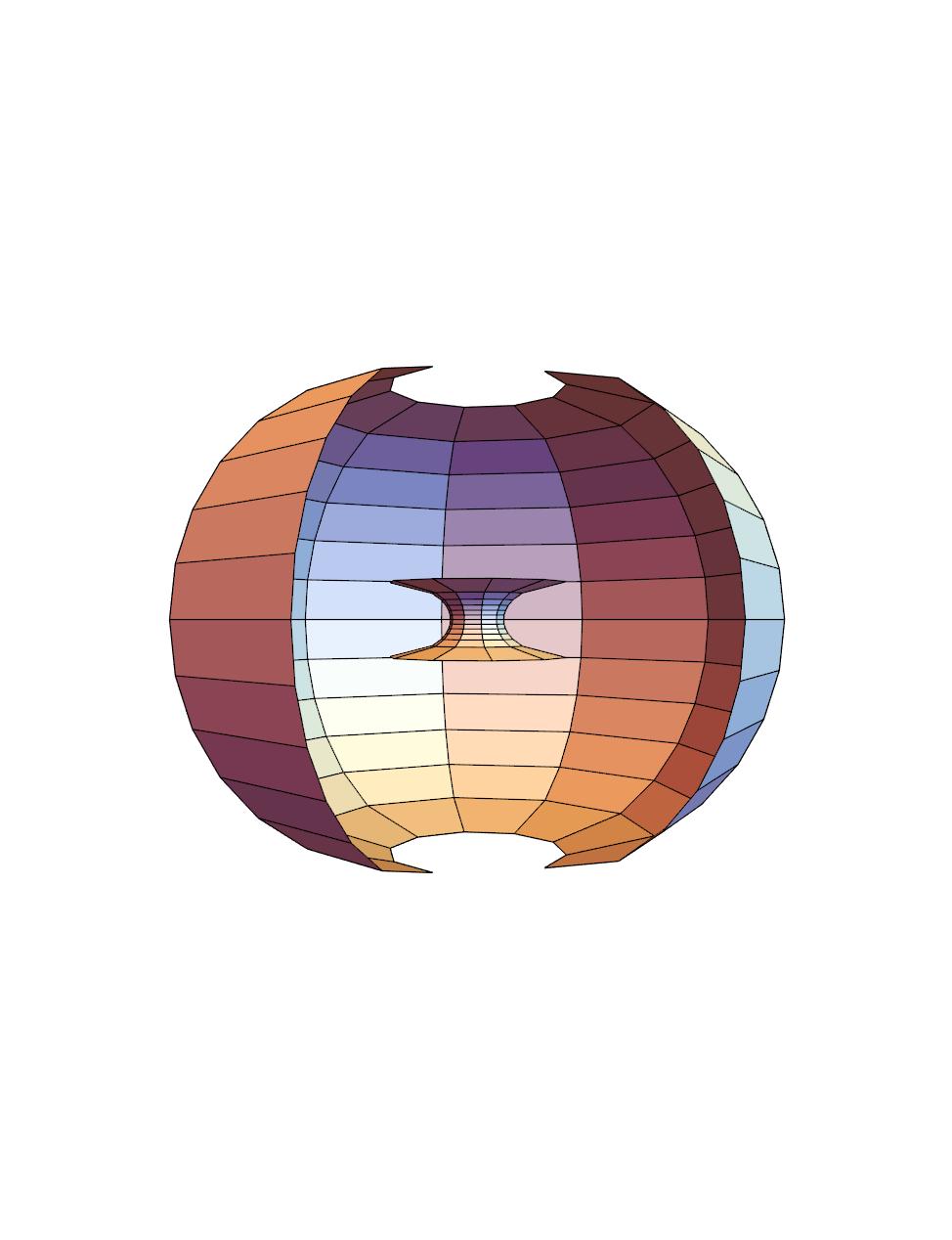}
\end{minipage}
\caption{Left: Discrete CMC surfaces which are parallel to the surface on the second from right of Figure \ref{fig:CGC1}. Right: Discrete CMC surfaces which are parallel to the surface on the right of Figure \ref{fig:CGC1}. Note that two pieces of surfaces in each case provide the fundamental building blocks for a complete discrete Delaunay surface, becoming embedded on the left and nonembedded on the right.} 
\label{fig:dela}
\end{figure}

\section{The Ricci flow on discrete surfaces of revolution}
We consider the discrete surface $x\colon \mathbb{Z}^2\times[0,\infty)\rightarrow \mathbb{R}^3$ with a flow parameter $t$ included:
\begin{equation}
x(m,n,t)=\Bigr(f(n,t)\cos \frac{2\pi m}{l}, f(n,t)\sin \frac{2\pi m}{l}, h(n,t)\Bigr), \label{surf2}
\end{equation}
and the unit normal vector $\nu$ can be written as
$$\nu(n,m,t)=\Bigr(a(n,t)\cos \frac{2\pi m}{l}, a(n,t)\sin \frac{2\pi m}{l}, b(n,t)\Bigr).$$
Using the metric terms $g_{11}, g_{22}$ and the Gaussian curvature $K$ for discrete surfaces of revolution  \eqref{surf2} as defined in Section $4$, we can consider the equations 
\begin{equation}
\left\{ \,
\begin{aligned}
&\frac{\partial}{\partial t}g_{11}(n,t)=-2K(n,t)g_{11}(n,t),\\
&\frac{\partial}{\partial t}g_{22}(n,t)=-2K(n,t)g_{22}(n,t)\;\;(n=0, \cdots ,k-1), \label{eq:flow3.5}
\end{aligned}
\right.
\end{equation}
where $k$ represents the number of layers above the $xy$-plane. Using the mixed area $A(\cdot,\cdot)$ defined in Section $4$, we can also consider  the normalized Ricci flow as
\begin{equation}
\left\{ \,
\begin{aligned}
&\frac{\partial}{\partial t}g_{11}(n,t)=(r(t)-2K(n,t))g_{11}(n,t),\\
&\frac{\partial}{\partial t}g_{22}(n,t)=(r(t)-2K(n,t))g_{22}(n,t)\;\;(n=0, \cdots ,k-1), \label{eq:flow4}
\end{aligned}
\right.
\end{equation}
where $$r(t):=\frac{\sum_{0\le i\le k-1}2K(i,t)A(x)(i,t)}{\sum_{0\le i\le k-1}A(x)(i,t)}.$$
We will add some conditions so that the solution $f(n,t)$, $h(n,t)$ can be determined, and we will now consider some particular cases of the normalized Ricci flow. Existence and uniqueness of these flows will be explained in Section $7$.\\
\\
\noindent \textbf{Normalized Ricci flow toward positive CGC surfaces with cones.}
\begin{equation}
\left\{ \,
\begin{aligned}
&\frac{\partial}{\partial t}g_{11}(n,t)=(r(t)-2K(n,t))g_{11}(n,t),\\
&\frac{\partial}{\partial t}g_{22}(n,t)=(r(t)-2K(n,t))g_{22}(n,t)\;\;(n=0, \cdots ,k-1),\\
&a(0,t)=1\;\;(b(0,t)=0),\\
&h(0,t)=0,\\
&f(k,t)=0.\label{eq:flow5}
\end{aligned}
\right.
\end{equation}
We want to define normalized Ricci flow toward positive CGC surfaces with cones, and we have the freedom to define a unit normal vector at one vertex. For the functions \eqref{eq:gaCGC1} which can parameterize the unit normal vector of the positive CGC surfaces with cones, when we choose $u_0=0$, we have $a(0)=1$ and $b(0)=0$. So we set $a(0,t)=1$ and $b(0,t)=0$ at $n=0$.  
The condition $h(0,t)=0$ is given to the function $h$ representing height, considering the symmetry with respect to the $xy$-plane. The condition $f(k,t)=0$ causes the top layer to consist of triangles instead of  quadrilaterals. The surfaces  can be considered to have a cone singularity. This means that the unit normal vector $\nu(m,k,t)$ is not fixed to one value and changes with respect to $m$. The initialized surface $x(m,n,0)$ is expected to flow toward a surface that has positive constant Gaussian curvature, which should be a sphere or a spindle type surface. Figure \ref{fig:fixf} shows numerical results (using Mathematica).  For the initialized surface in Figure \ref{fig:fixf}, we uses the parameter in \cite{dumbbell} and its surface is half the shape of a dumbbell. The left side of Figure \ref{fig:gauss1} shows that the Gaussian curvatures are approaching roughly 1.05474  fairly quickly in Figure \ref{fig:fixf}. Then we consider the constant Gaussian curvature $K=1.0547444492811$ and substitute $c=1.0547444492811$ into \eqref{eq:CGC} and compare that with $f(\cdot,10^4)$ and $h(\cdot,10^4)$ obtained in Figure \ref{fig:fixf}. We choose $u_n$ and $p$ such that $p\cos(\sqrt{c}u_{n})=f(n,10^4)$. Substituting these  $u_n$ and $p$ into in \eqref{eq:CGC}, we have
\begin{equation*}
\begin{aligned}
&(h(0),h(1),h(2),h(3),h(4),h(5),h(6))\\
&\approx(0, 0.455256, 0.738473, 0.874059, 0.940356, 0.978055, 1.00206).
\end{aligned}
\end{equation*}
The $h(\cdot,10^4)$ obtained in Figure \ref{fig:fixf} are
\begin{equation*}
\begin{aligned}
&(h(0,10^4),h(1,10^4),h(2,10^4),h(3,10^4),h(4,10^4),h(5,10^4),h(6,10^4))\\
&\approx(0, 0.455256, 0.738473, 0.874059, 0.940356, 0.978055, 1.00206).
\end{aligned}
\end{equation*}
Numerically, we can see that $f(\cdot,t)$ and $h(\cdot,t)$ are approaching the functions for a CGC surface of revolution.

\begin{figure}[htbp]
 \hspace{-1.0cm}
 \centering
\begin{minipage}[t]{0.16\textwidth}
 \includegraphics[trim= 2.0cm 0cm 2.0cm 0cm, clip, width=2.5cm]{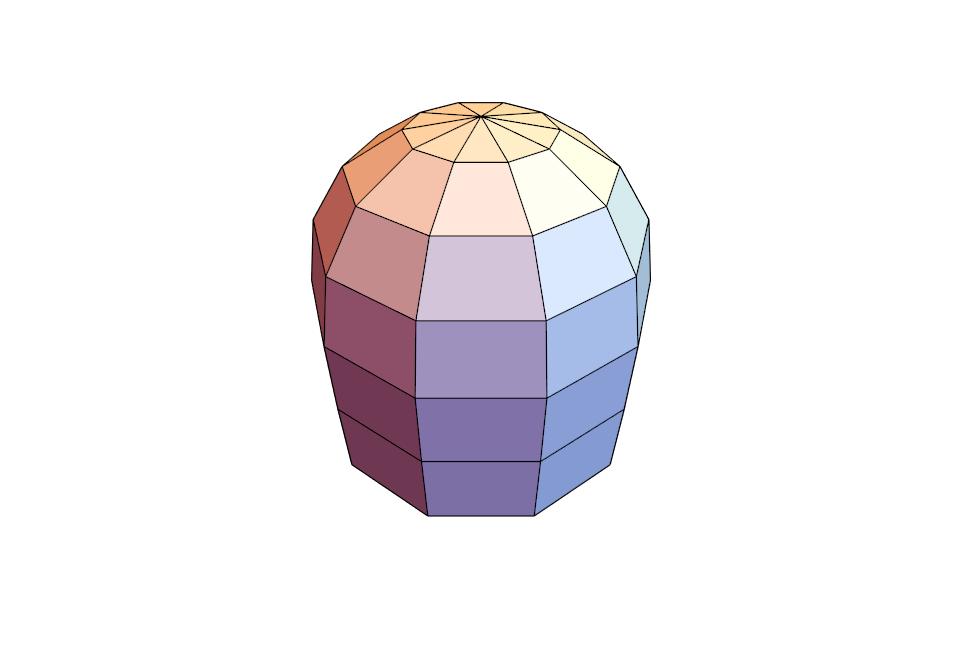} 
\end{minipage}
\begin{minipage}[t]{0.16\textwidth}
 \includegraphics[trim= 2.0cm 0cm 2.0cm 0cm, clip, width=2.5cm]{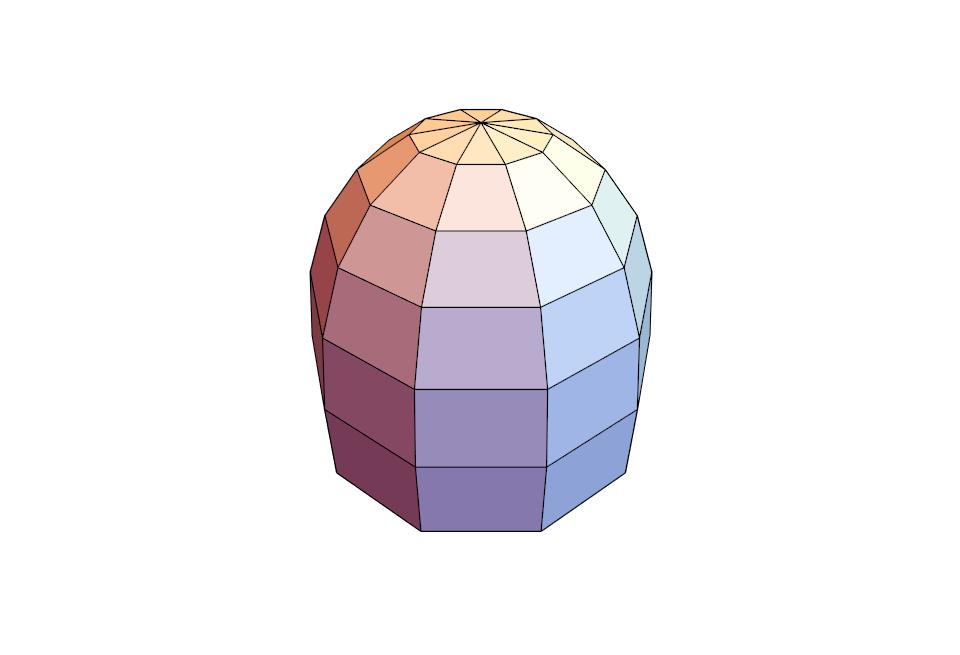} 
\end{minipage}
\begin{minipage}[t]{0.16\textwidth}
 \includegraphics[trim= 2.0cm 0cm 2.0cm 0cm, clip, width=2.5cm]{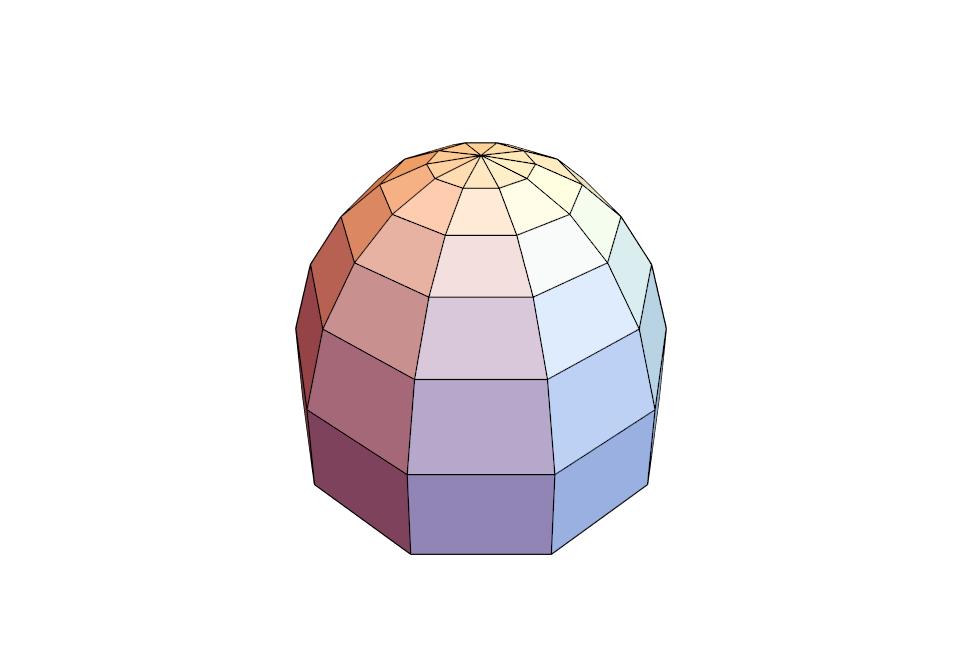}
\end{minipage}
\begin{minipage}[t]{0.16\textwidth}
 \includegraphics[trim= 2.0cm 0cm 2.0cm 0cm, clip, width=2.5cm]{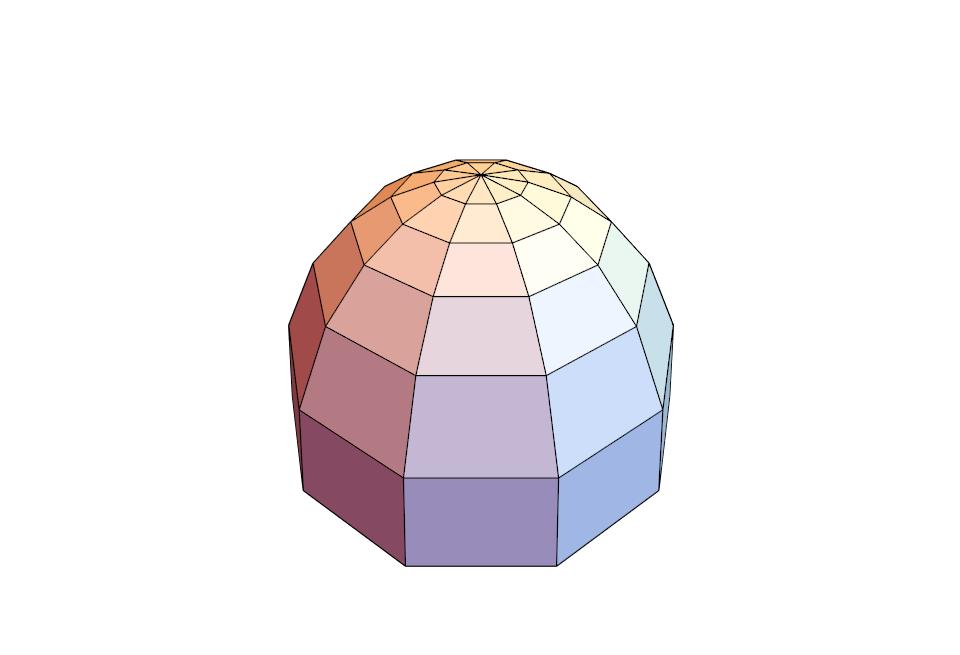}
\end{minipage}
\begin{minipage}[t]{0.16\textwidth}
 \includegraphics[trim= 2.0cm 0cm 2.0cm 0cm, clip, width=2.5cm]{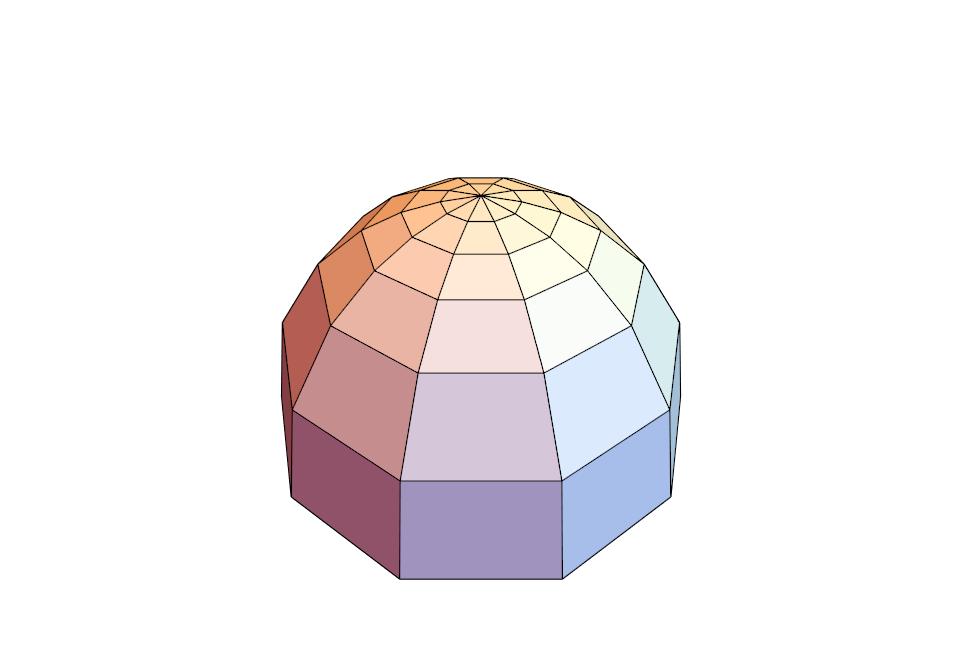}
\end{minipage}
\begin{minipage}[t]{0.16\textwidth}
 \includegraphics[trim= 2.0cm 0cm 2.0cm 0cm, clip, width=2.5cm]{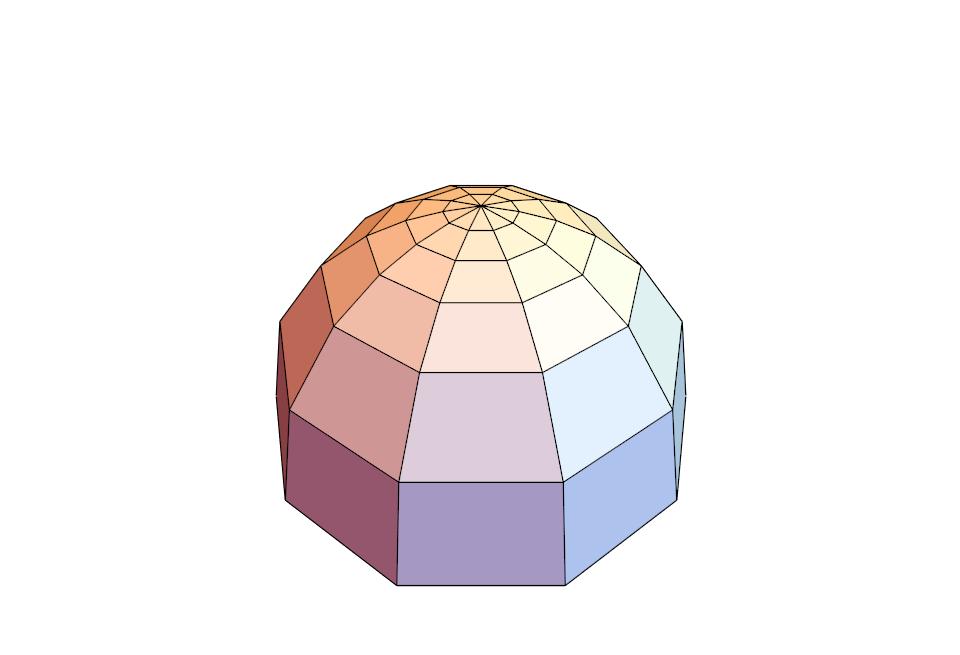}
\end{minipage}
\caption{A discrete surface moving under normalized Ricci flow \eqref{eq:flow5} toward a positive CGC surface with a cone, from left to right.}
\label{fig:fixf}
\end{figure}

\begin{figure}[htbp]
 \centering
\begin{minipage}[t]{0.48\textwidth}
 \includegraphics[width=4.5cm]{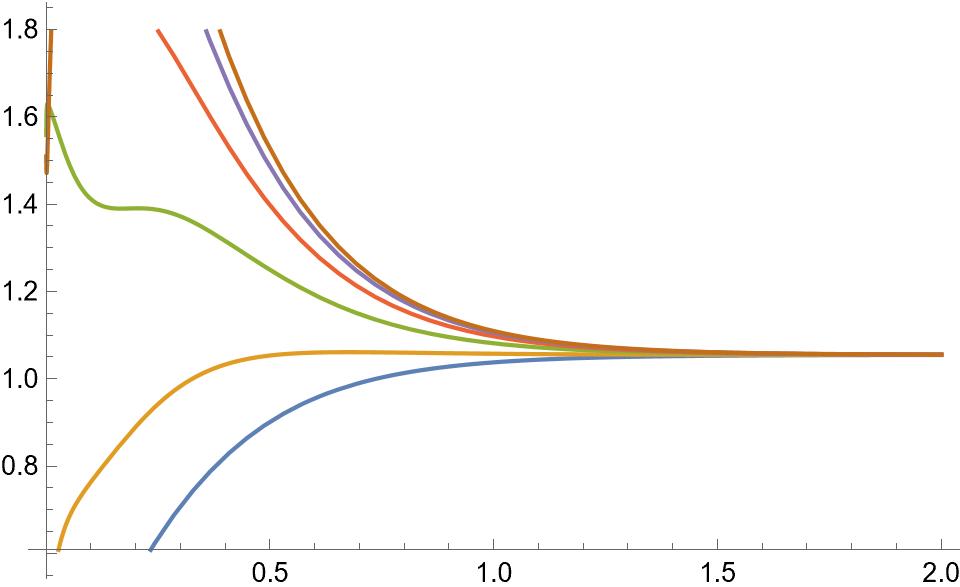} 
\end{minipage}
\begin{minipage}[t]{0.48\textwidth}
 \includegraphics[width=4.5cm]{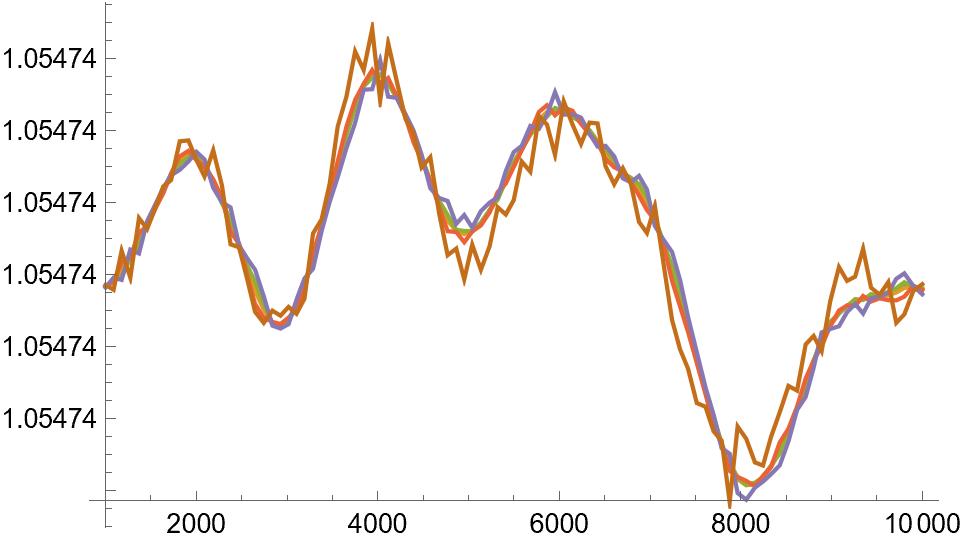} 
\end{minipage}
\caption{Left: The Gaussian curvature from time $0$ to $2$ on each layer. Right: The Gaussian curvature from time $10^3$ to $10^4$ on each layer (the layers are color coded.}
\label{fig:gauss1}
\end{figure}

\noindent \textbf{Normalized Ricci flow toward negative CGC surfaces with cones.}
\begin{equation}
\left\{ \,
\begin{aligned}
&\frac{\partial}{\partial t}g_{11}(n,t)=(r(t)-2K(n,t))g_{11}(n,t),\\
&\frac{\partial}{\partial t}g_{22}(n,t)=(r(t)-2K(n,t))g_{22}(n,t)\;\;(n=0, \cdots ,k-1),\\
&b(0,t)=1\;\;(a(0,t)=0),\\
&h(0,t)=0,\\
&f(k,t)=0.\label{eq:flow6}
\end{aligned}
\right.
\end{equation}
Here, we want to define normalized Ricci flow toward negative CGC surfaces with cones. For the functions \eqref{eq:gaCGC-1h}, there exists $u_0$ such that $a(0)=0$ and $b(0)=1$. Therefore we set $a(0,t)=0$ and $b(0,t)=1$ at $n=0$. We added the condition $h(0,t)=0$ and $f(k,t)=0$ for the same reason as in \eqref{eq:flow5}. This flow is expected to converge to a surface with negative constant Gaussian curvature, and Figure \ref{fig:negf} shows numerical results.

\begin{figure}[htbp]
 \hspace{-1.0cm}
 \centering
\begin{minipage}[t]{0.16\textwidth}
 \includegraphics[trim= 2.0cm 0cm 2.0cm 0cm, clip, width=2.2cm]{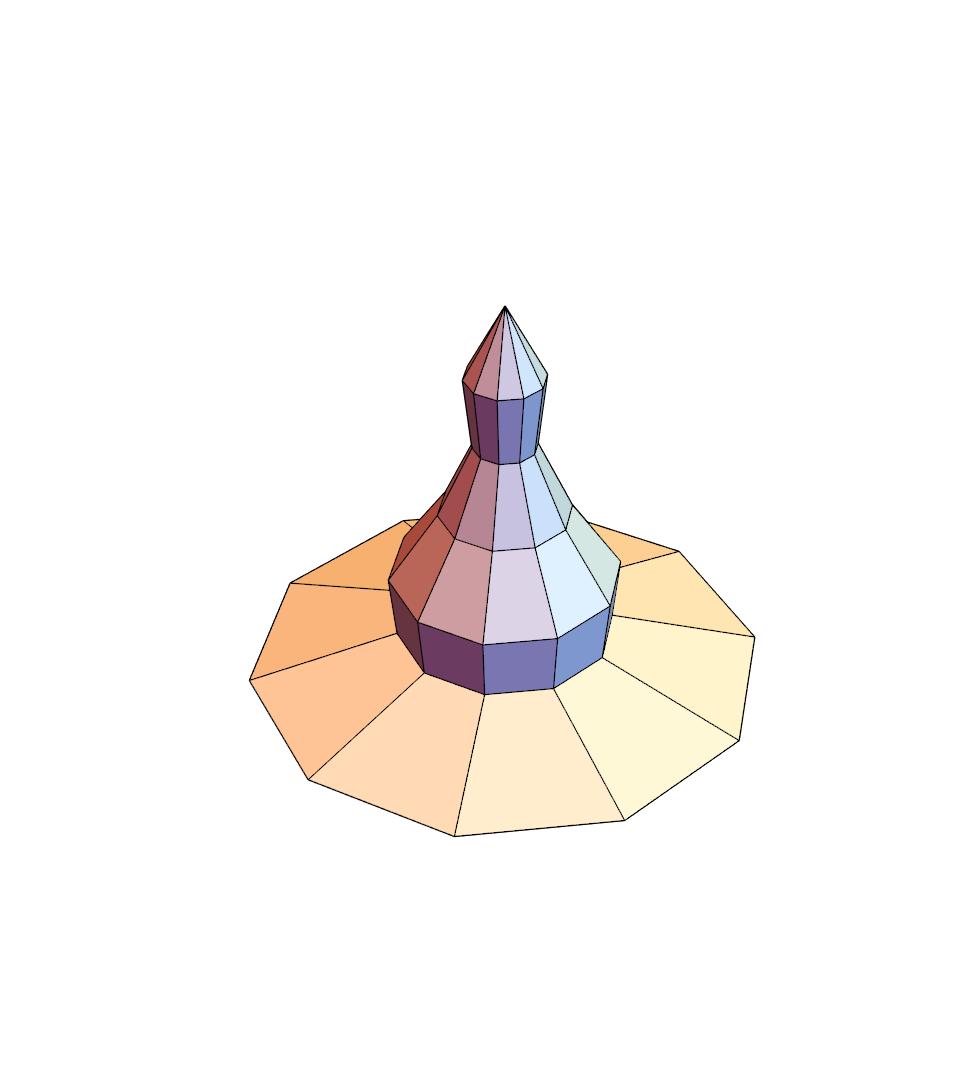} 
\end{minipage}
\vspace{-0.5cm}
\begin{minipage}[t]{0.16\textwidth}
 \includegraphics[trim= 2.0cm 0cm 2.0cm 0cm, clip, width=2.2cm]{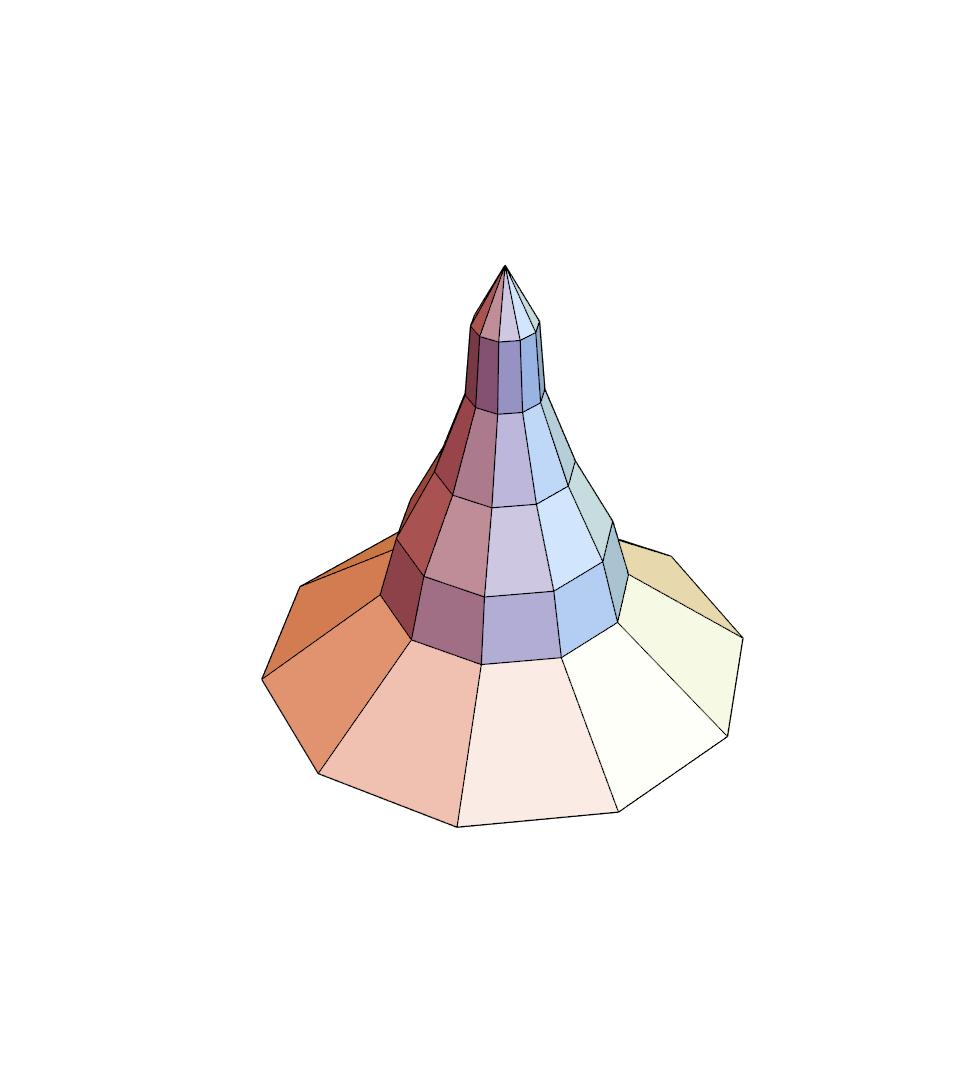} 
\end{minipage}
\begin{minipage}[t]{0.16\textwidth}
 \includegraphics[trim= 2.0cm 0cm 2.0cm 0cm, clip, width=2.2cm]{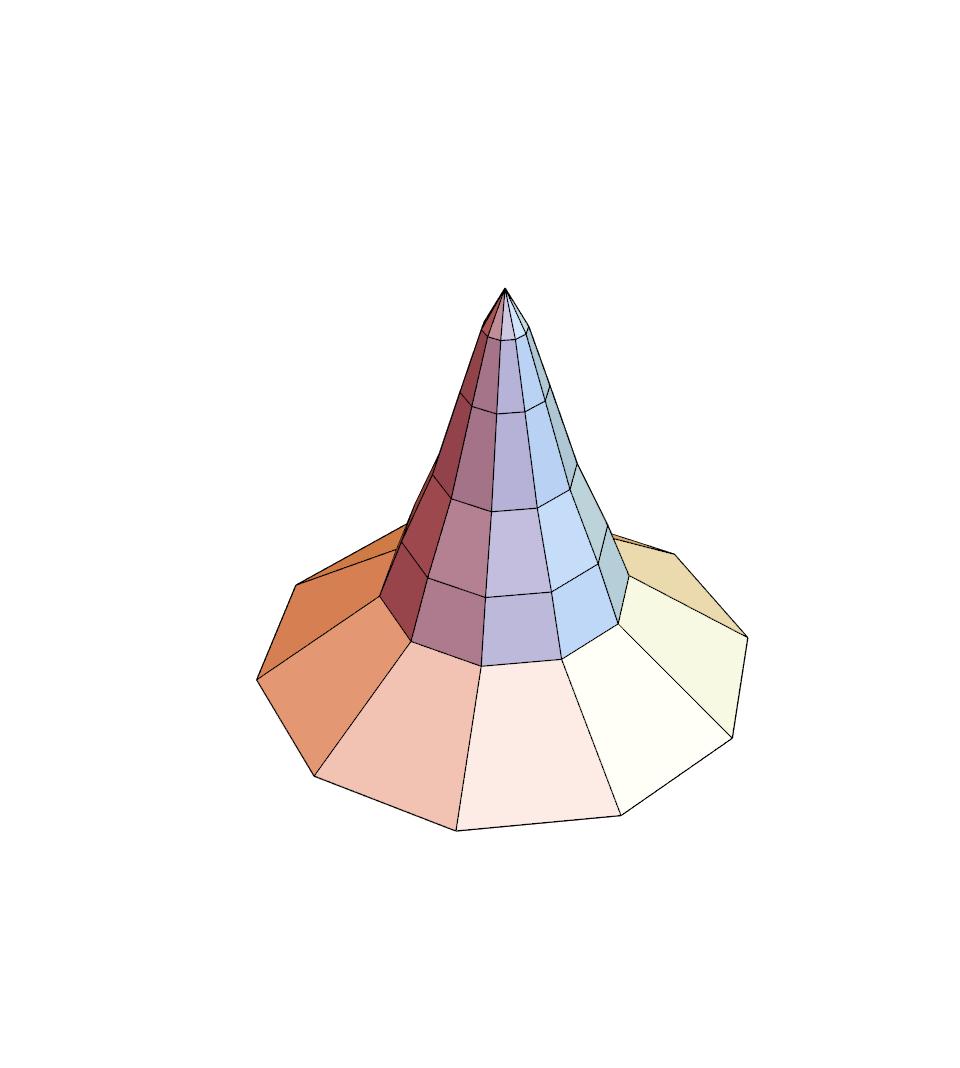}
\end{minipage}
\begin{minipage}[t]{0.16\textwidth}
 \includegraphics[trim= 2.0cm 0cm 2.0cm 0cm, clip, width=2.2cm]{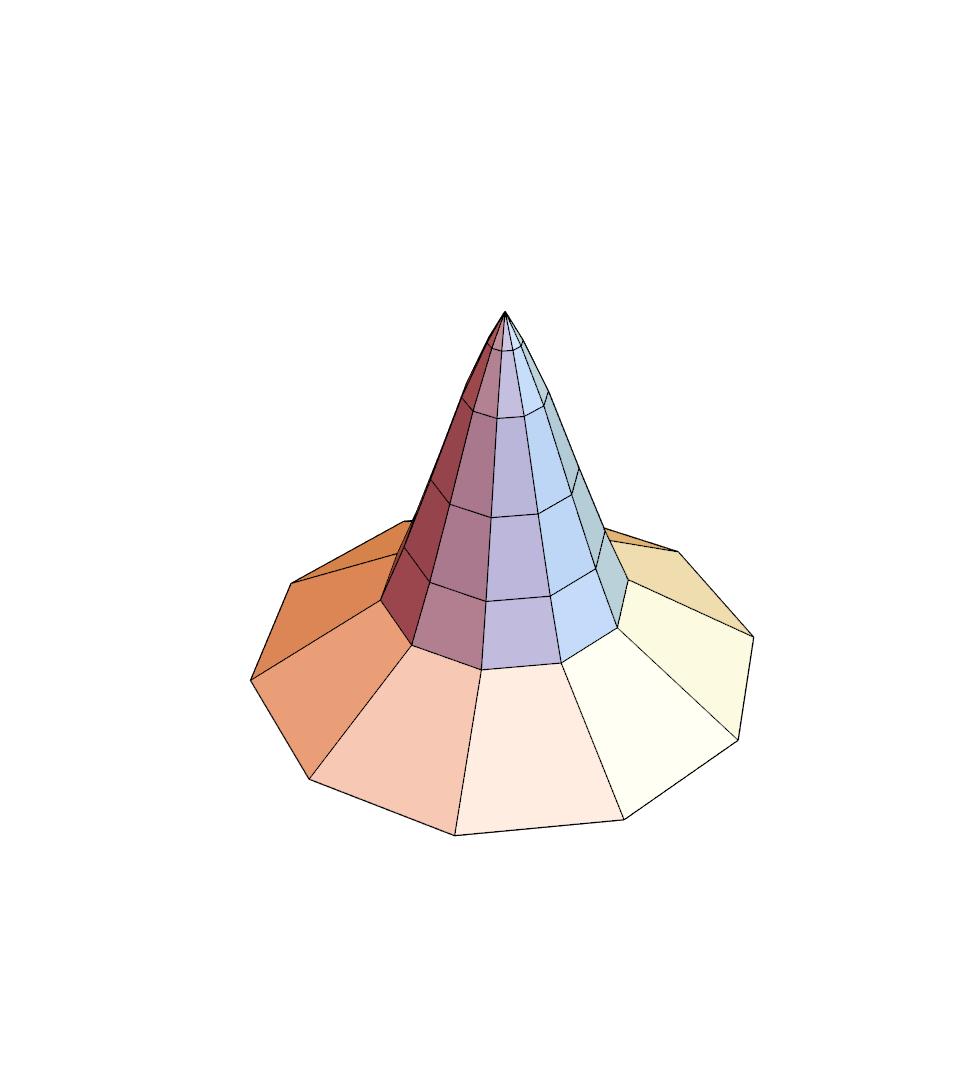}
\end{minipage}
\begin{minipage}[t]{0.17\textwidth}
 \includegraphics[trim= 2.0cm 0cm 2.0cm 0cm, clip, width=2.2cm]{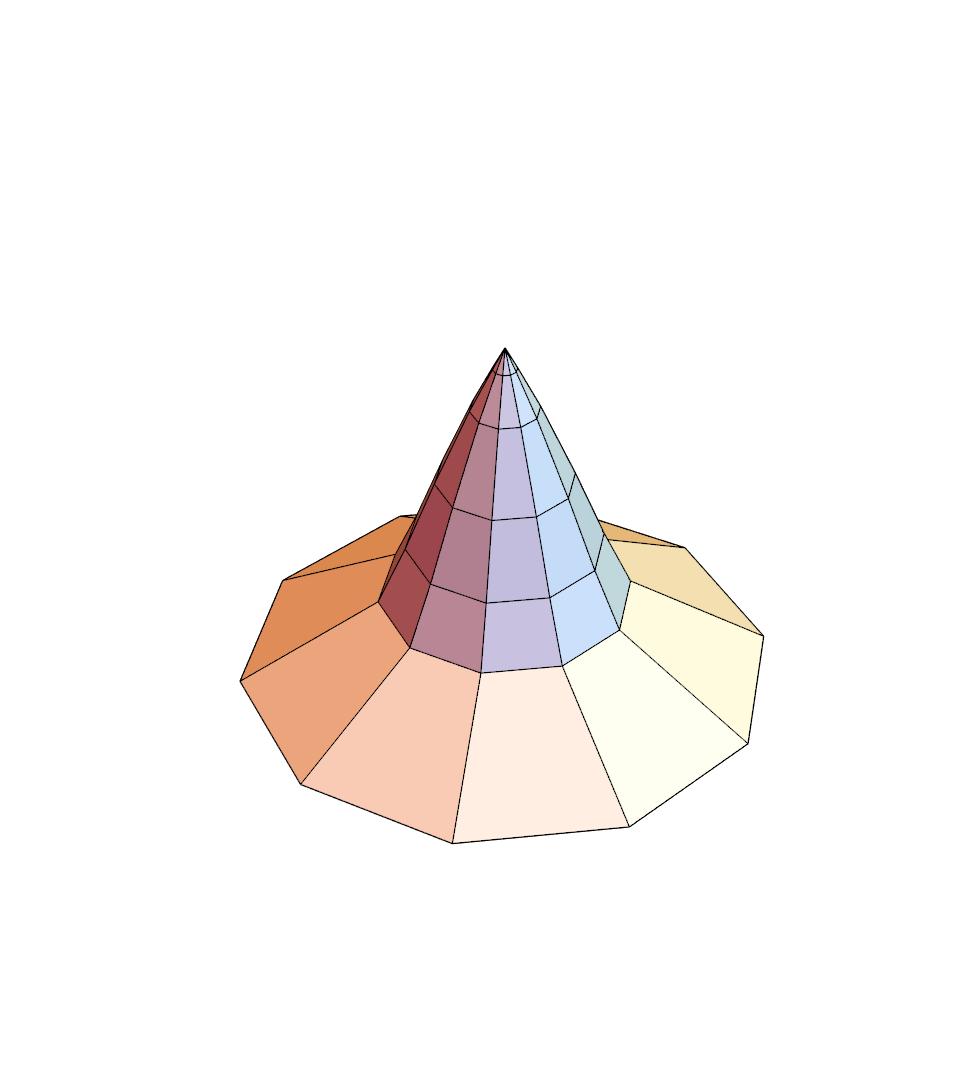}
\end{minipage}
\begin{minipage}[t]{0.16\textwidth}
 \includegraphics[trim= 2.0cm 0cm 2.0cm 0cm, clip, width=2.2cm]{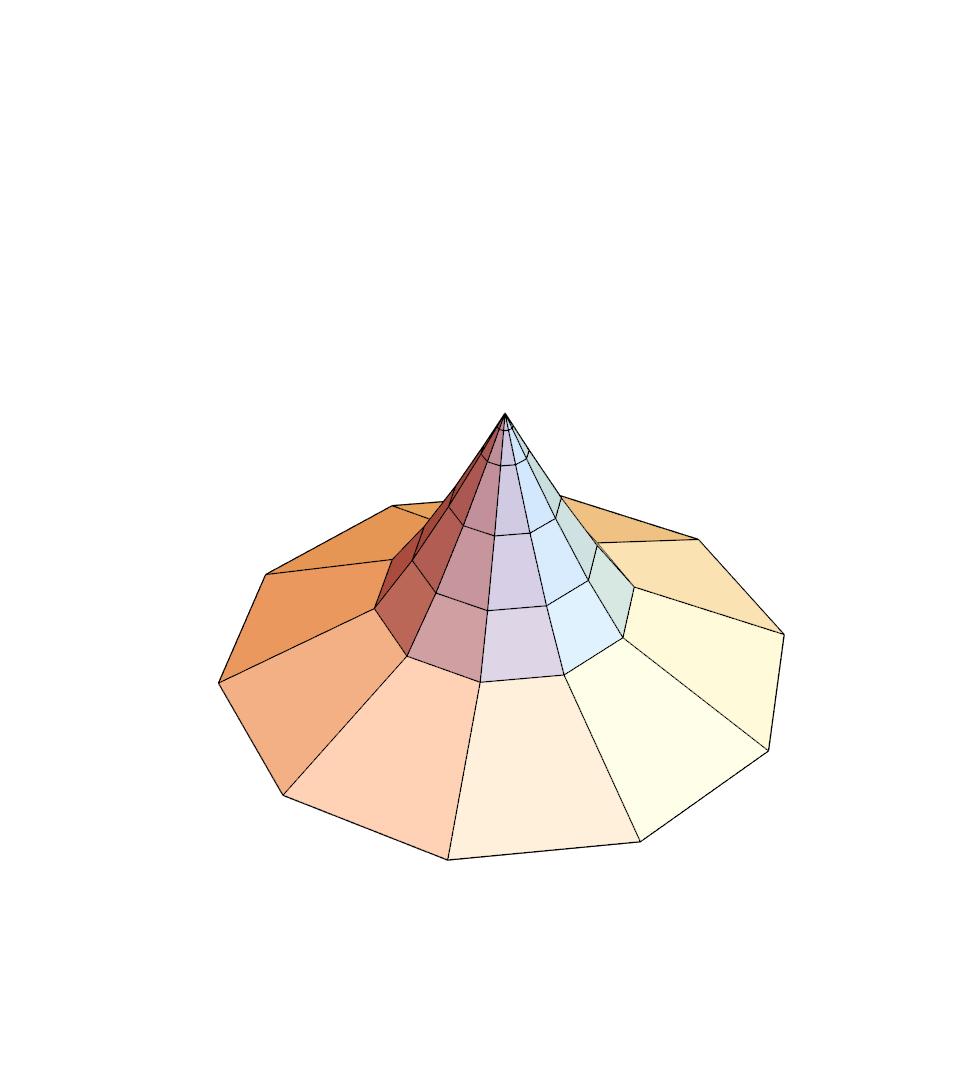}
\end{minipage}
\caption{A discrete surface moving under normalized Ricci flow \eqref{eq:flow6} toward a negative CGC surface with a cone, from left to right..}
\label{fig:negf}
\end{figure}

\noindent \textbf{Normalized Ricci flow toward positive CGC surfaces with cusps.}
\begin{equation}
\left\{ \,
\begin{aligned}
&\frac{\partial}{\partial t}g_{11}(n,t)=(r(t)-2K(n,t))g_{11}(n,t),\\
&\frac{\partial}{\partial t}g_{22}(n,t)=(r(t)-2K(n,t))g_{22}(n,t)\;\;(n=0, \cdots ,k-1),\\
&a(0,t)=1\;\;(b(0,t)=0),\\
&h(0,t)=0,\\
&b(k,t)=1\;\;(a(k,t)=0). \label{eq:flow7}
\end{aligned}
\right.
\end{equation}
Here, we want to define normalized Ricci flow toward positive CGC sufaces with cusps. From the functions \eqref{eq:gaCGC1}, we set $a(0,t)=1$ and $b(0,t)=0$ for the same reason as in \eqref{eq:flow5}. We added the condition $h(0,t)=0$ for the same reason as in \eqref{eq:flow5}.
The condition $b(k,t)=1$ is added to avoid triangulation as in \eqref{eq:flow5} and \eqref{eq:flow6}. Since there exists $u_k$ such that $a(k)=0$ and $b(k)=1$ for the functions \eqref{eq:gaCGC1}, it is natural to set $b(k,t)=1$. The initialized surface $x(m,n,0)$ is expected to move toward a sphere or a barrel type CGC surface which has a cusp at the top. Figure \ref{fig:fixa} shows numerical results. We also can again confirm that $f(\cdot,t)$ and $h(\cdot,t)$ are approaching the functions for a CGC surface of revolution.

\begin{figure}[htbp]
 \hspace{-0.5cm}
 \centering
\begin{minipage}[t]{0.16\textwidth}
 \includegraphics[trim= 2.0cm 0cm 1.9cm 0cm, clip, width=2.0cm]{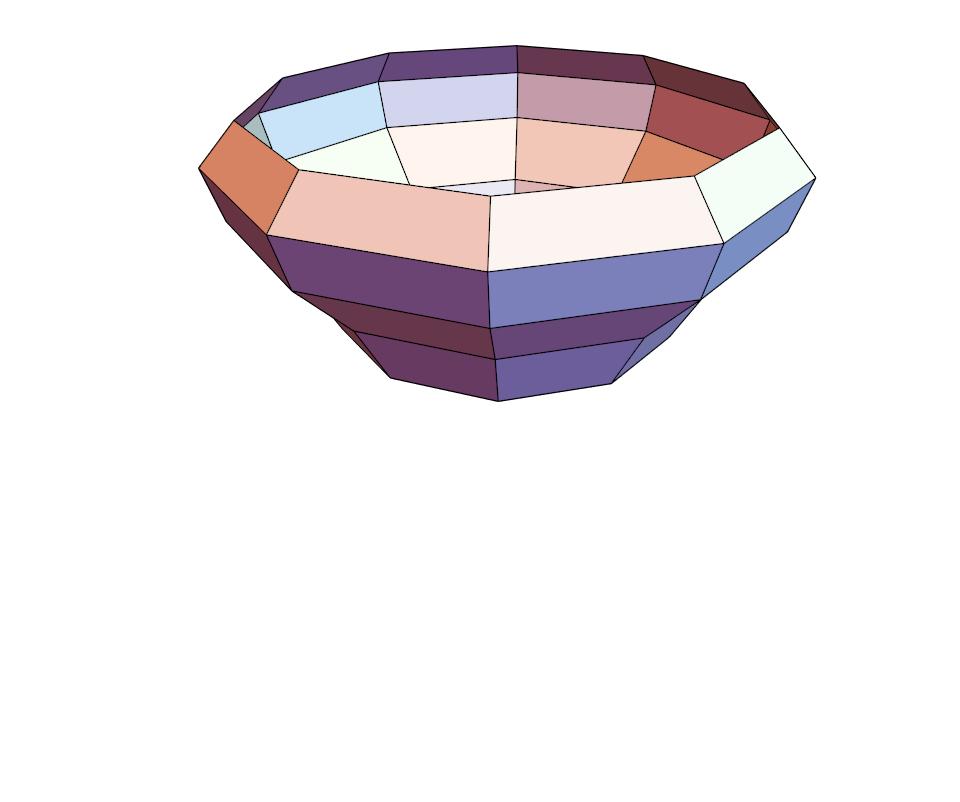} 
\end{minipage}
 \vspace{-0.7cm}
\begin{minipage}[t]{0.16\textwidth}
 \includegraphics[trim= 2.0cm 0cm 2.0cm 0cm, clip, width=2.0cm]{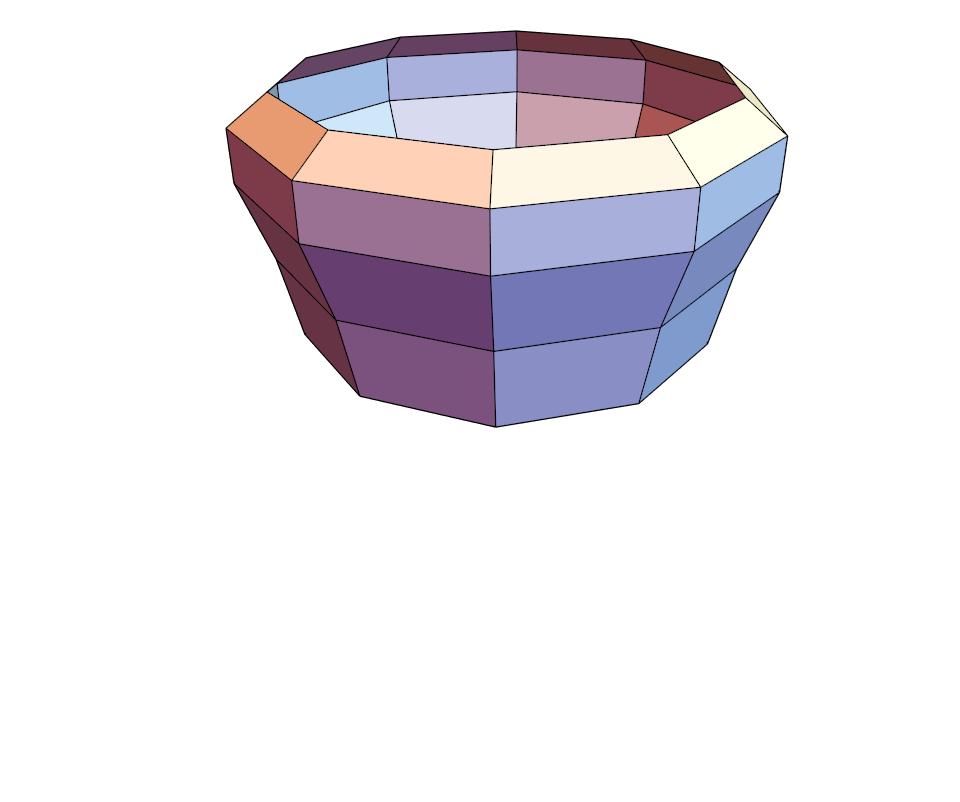} 
\end{minipage}
\begin{minipage}[t]{0.16\textwidth}
 \includegraphics[trim= 2.0cm 0cm 2.0cm 0cm, clip, width=2.0cm]{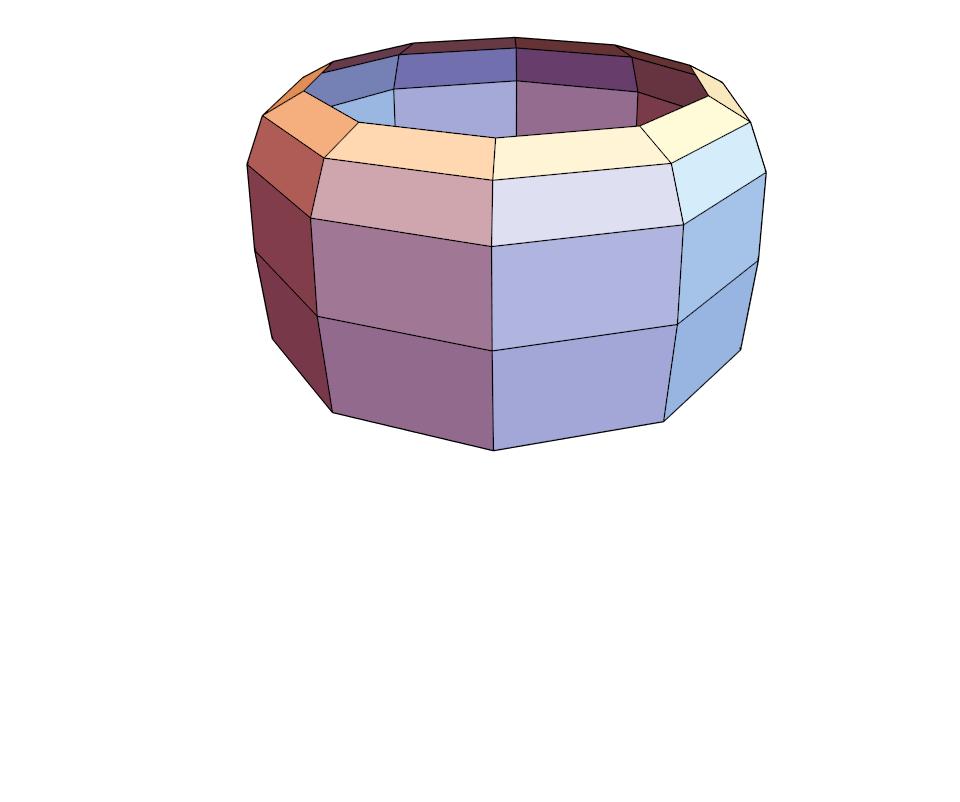}
\end{minipage}
\begin{minipage}[t]{0.16\textwidth}
 \includegraphics[trim= 2.0cm 0cm 2.0cm 0cm, clip, width=2.0cm]{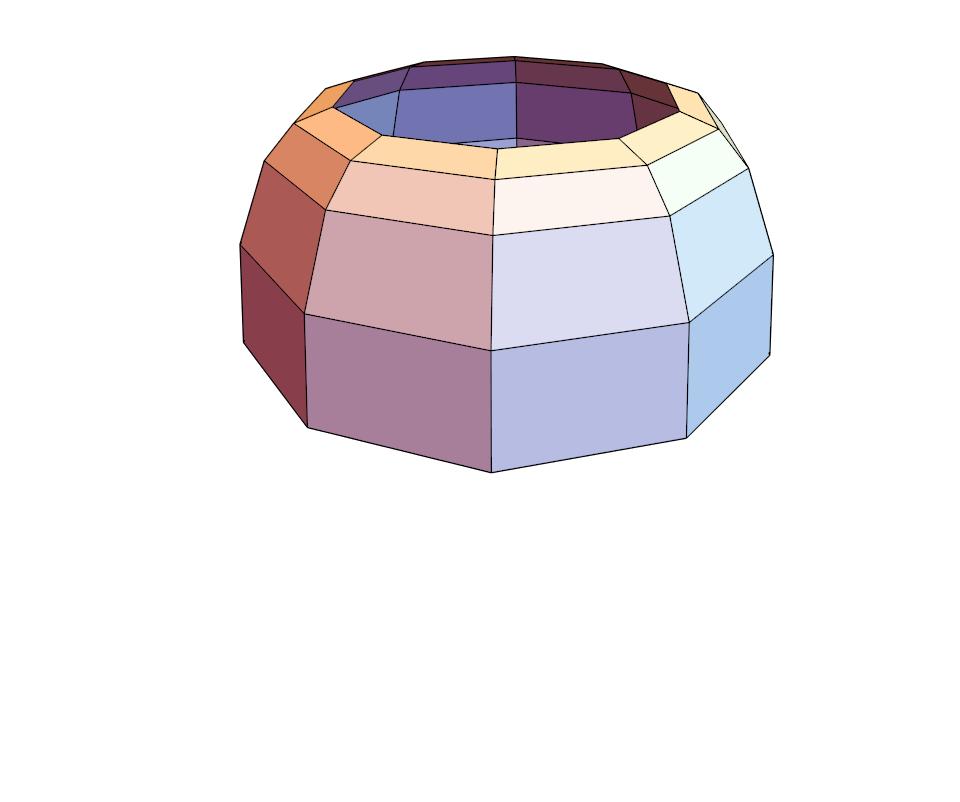}
\end{minipage}
\begin{minipage}[t]{0.16\textwidth}
 \includegraphics[trim= 2.0cm 0cm 2.0cm 0cm, clip, width=2.0cm]{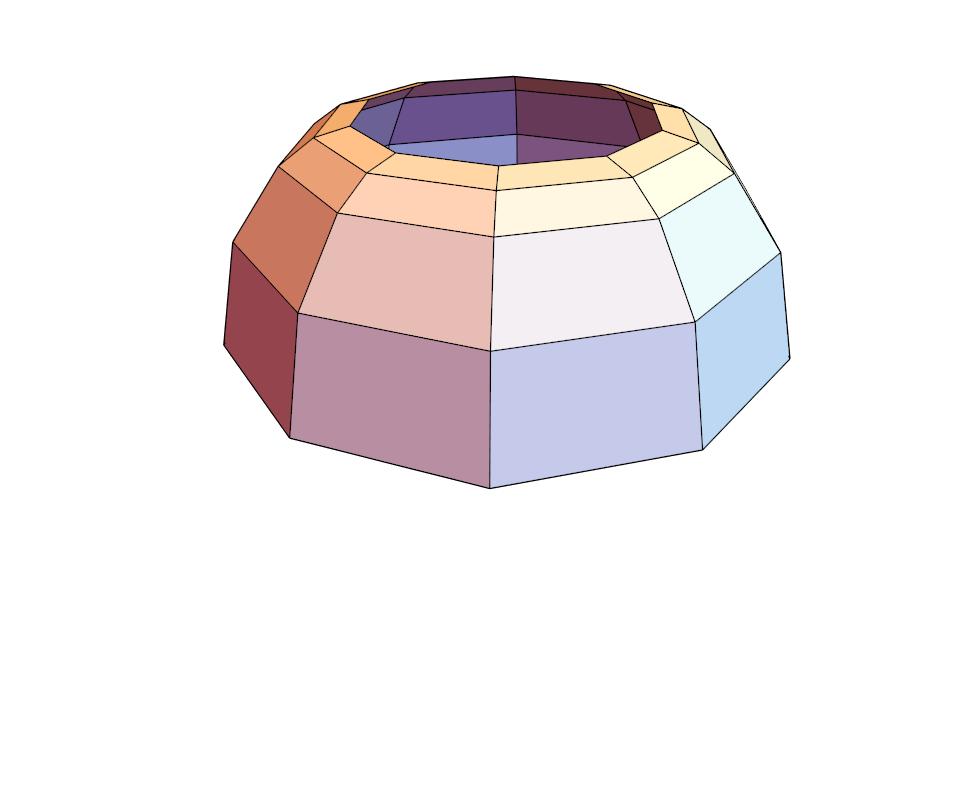}
\end{minipage}
\begin{minipage}[t]{0.16\textwidth}
 \includegraphics[trim= 2.0cm 0cm 2.0cm 0cm, clip, width=2.0cm]{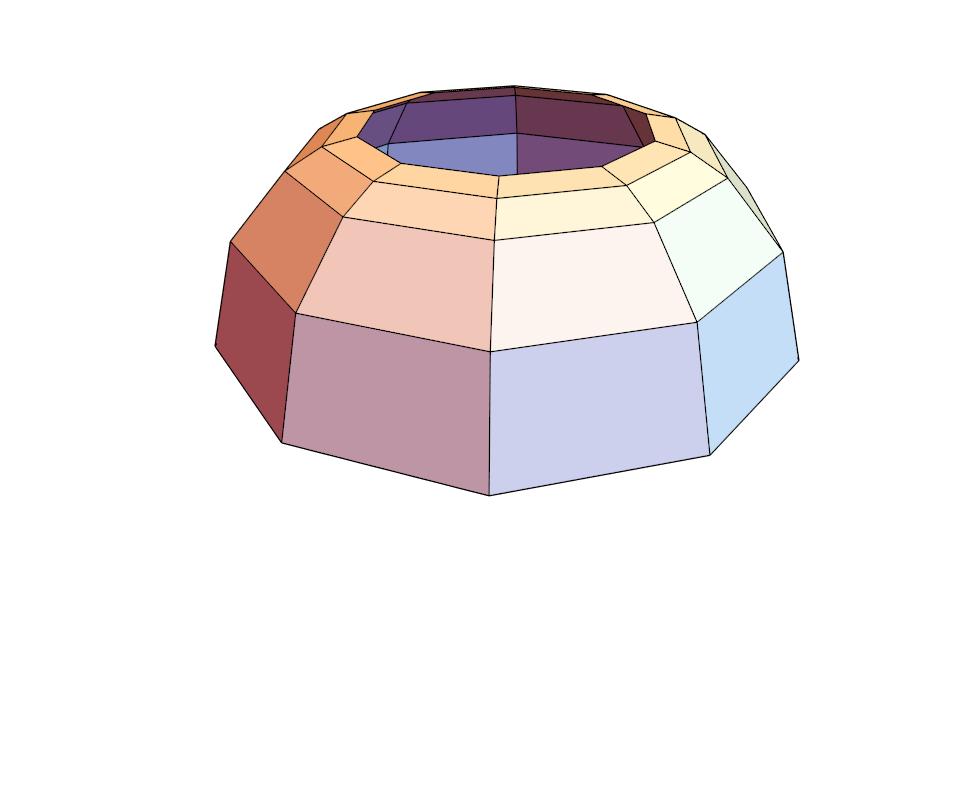}
\end{minipage}
\caption{A discrete surface moving under the normalized Ricci flow \eqref{eq:flow7} toward a positive CGC surface with a cusp, from left to right.}
\label{fig:fixa}
\end{figure}

\noindent \textbf{Normalized Ricci flow toward negative CGC surfaces with cusps.}
\begin{equation}
\left\{ \,
\begin{aligned}
&\frac{\partial}{\partial t}g_{11}(n,t)=(r(t)-2K(n,t))g_{11}(n,t),\\
&\frac{\partial}{\partial t}g_{22}(n,t)=(r(t)-2K(n,t))g_{22}(n,t)\;\;(n=0, \cdots ,k-1),\\
&a(0,t)=1\;\;(b(0,t)=0),\\
&h(0,t)=0,\\
&b(k,t)=-1\;\;(a(k,t)=0). \label{eq:flow8}
\end{aligned}
\right.
\end{equation}
Here, we want to define normalized Ricci flow toward negative CGC sufaces with cusps. For the functions \eqref{eq:gaCGC-1c}, when we choose $u_0=0$, we have $a(0)=1$ and $b(0)=0$. So we set $a(0,t)=1$ and $b(0,t)=0$ at $n=0$. We add the condition $h(0,t)=0$ for the same reason as in \eqref{eq:flow5}.
The condition $b(k,t)=-1$ is also added for the same reason as in \eqref{eq:flow7}. The initialized surface $x(m,n,0)$ is expected to move toward an anvil-shaped CGC surface which has a cusp at the top. Figure \ref{fig:nega} shows numerical results.

\begin{figure}[htbp]
 \hspace{-0.5cm}
 \centering
\begin{minipage}[t]{0.16\textwidth}
 \includegraphics[width=2.5cm]{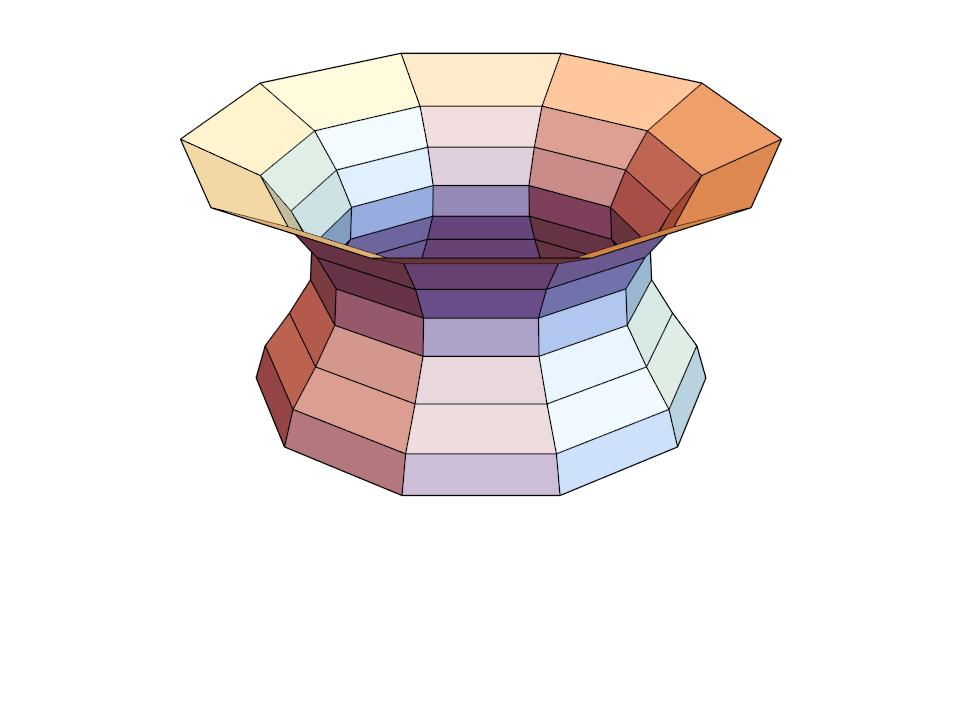} 
\end{minipage}
\vspace{-0.3cm}
\begin{minipage}[t]{0.16\textwidth}
 \includegraphics[width=2.5cm]{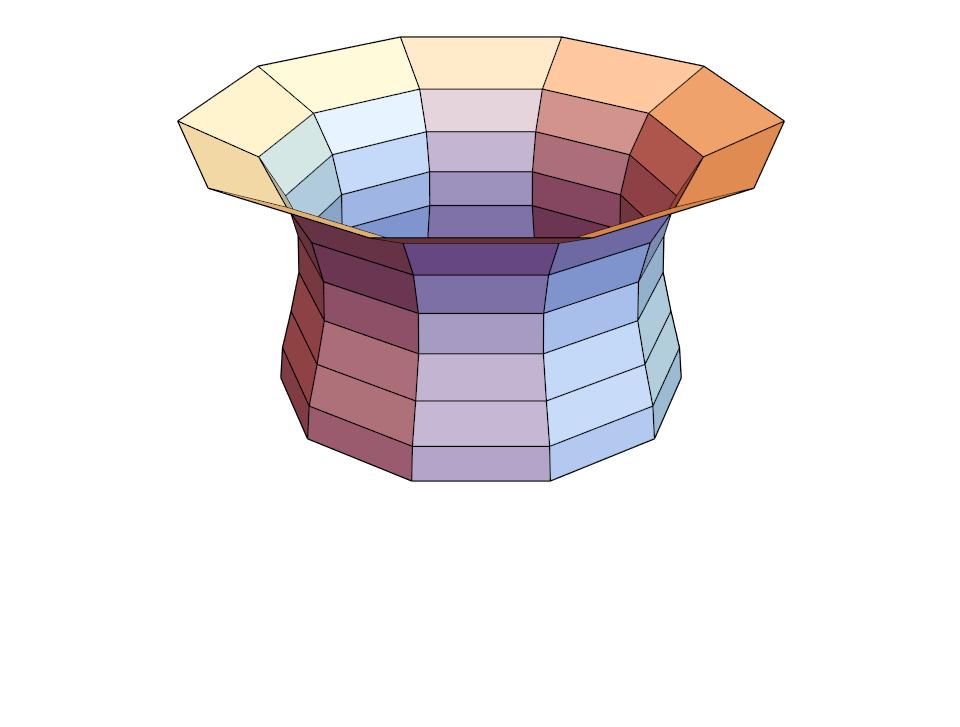} 
\end{minipage}
\begin{minipage}[t]{0.16\textwidth}
 \includegraphics[width=2.5cm]{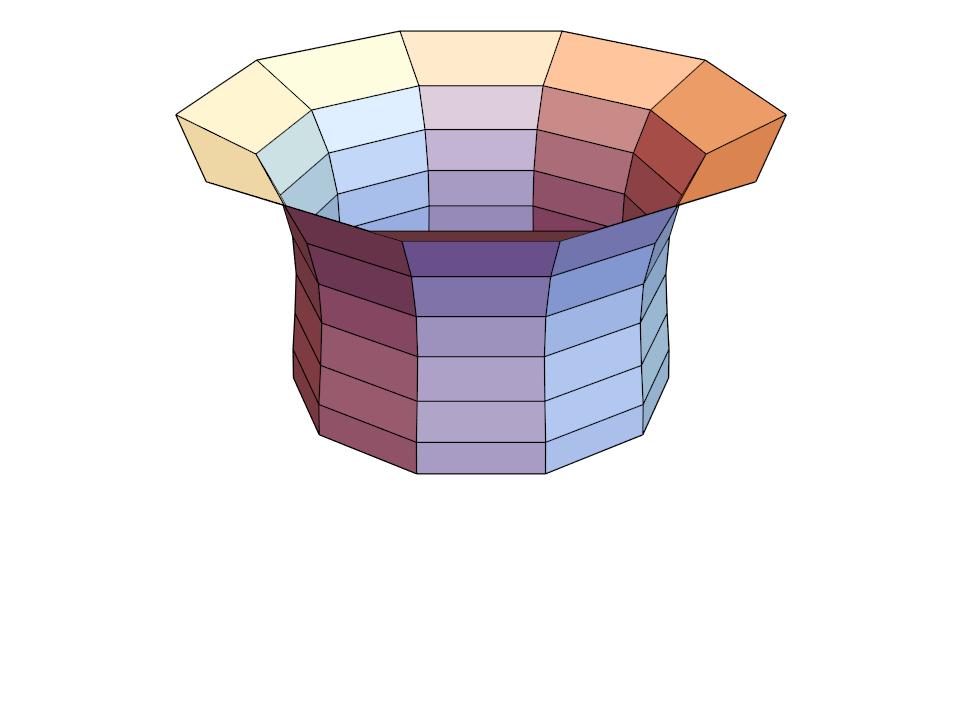}
\end{minipage}
\begin{minipage}[t]{0.17\textwidth}
 \includegraphics[width=2.5cm]{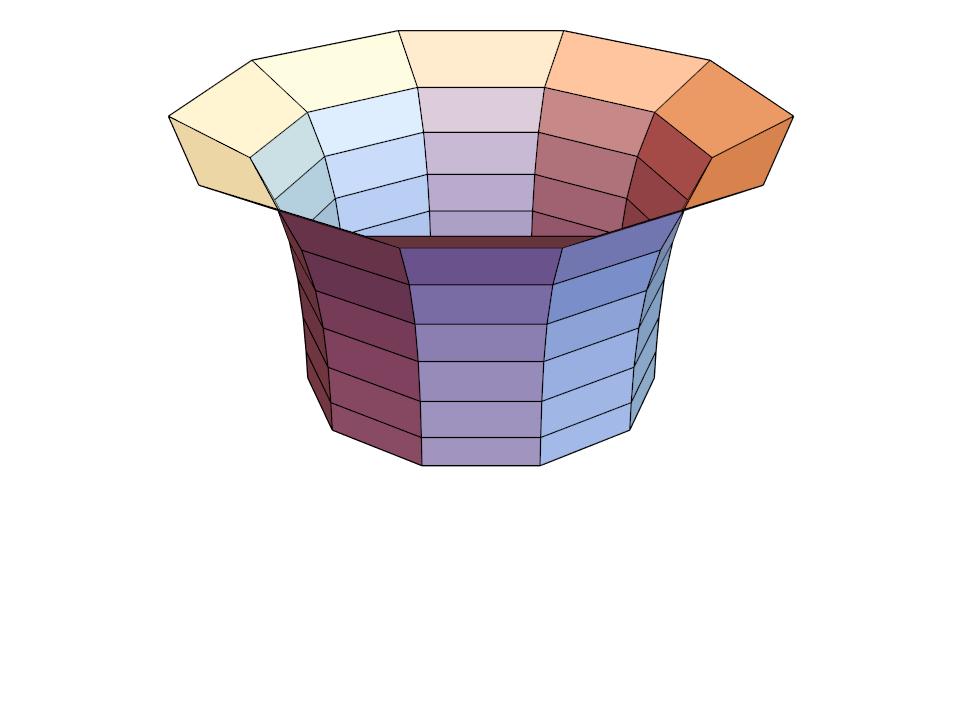}
\end{minipage}
\begin{minipage}[t]{0.17\textwidth}
 \includegraphics[width=2.5cm]{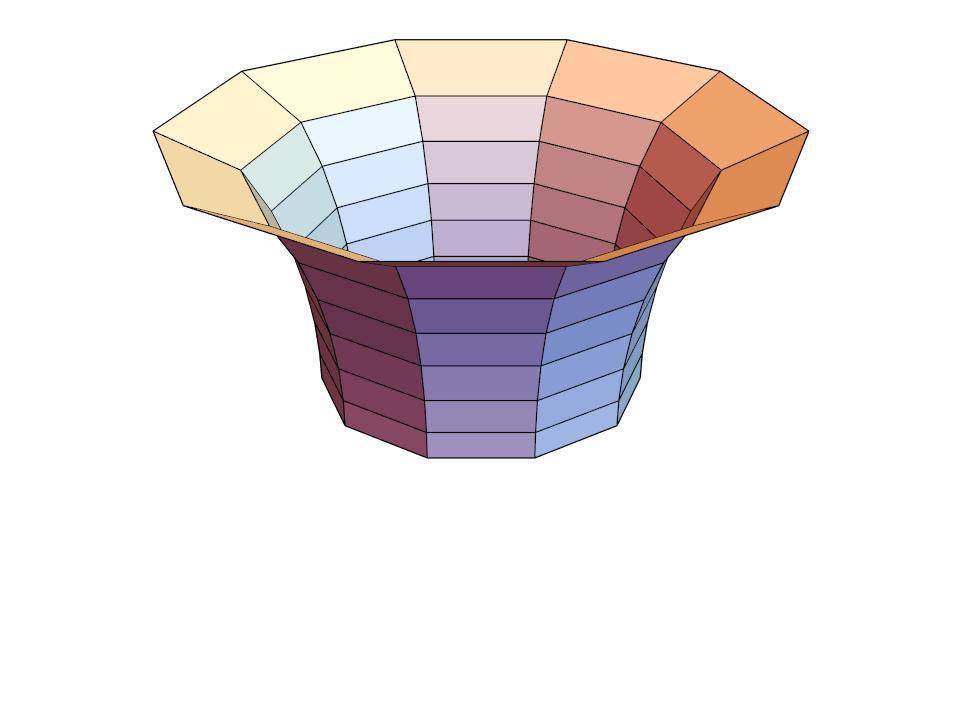}
\end{minipage}
\begin{minipage}[t]{0.16\textwidth}
 \includegraphics[width=2.5cm]{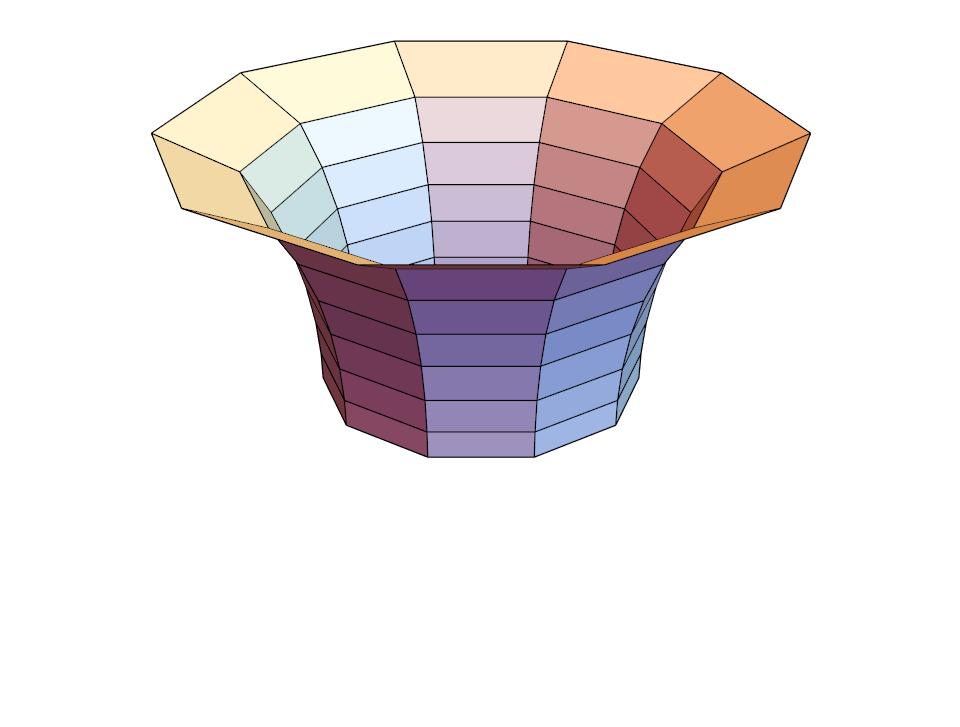}
\end{minipage}
\caption{A discrete surface moving under the normalized Ricci flow \eqref{eq:flow8} toward a negative CGC surface with a cusp, from left to right.}
\label{fig:nega}
\end{figure}

\section{Existence and uniqueness of normalized Ricci flow}
First, we show that the normalized Ricci flow preserves area in the discrete case as well.
\begin{lem}
If $f(n,t)$, $h(n,t)$ satisfy any of \eqref{eq:flow5}, \eqref{eq:flow6}, \eqref{eq:flow7} or \eqref{eq:flow8}, then the area $\sum_{0\le i\le k-1}A(x)(i,t)$ is constant with respect to $t$. \label{lem:area}
\end{lem}
\begin{proof}
Since 
$$A(x)(i,t)=\sqrt{g_{11}(i,t)g_{22}(i,t)},$$
we have
$$\frac{d}{dt}A(x)(i,t)=\frac{g_{11}\frac{d}{dt}g_{22}+g_{22}\frac{d}{dt}g_{11}}{2\sqrt{g_{11}g_{22}}}=(r(t)-2K(i,t))A(x)(i,t).$$
Therefore,
\begin{equation*}
\begin{split}
\frac{d}{dt}\sum_{0\le i\le k-1}A(x)(i,t)&=\sum_{0\le i\le k-1}(r(t)-2K(i,t))A(x)(i,t)\\ &=\frac{\sum_{0\le i\le k-1}2K(i,t)A(x)(i,t)}{\sum_{0\le i\le k-1}A(x)(i,t)}\sum_{0\le i\le k-1}A(x)(i,t) \\
&\quad-\sum_{0\le i\le k-1}2K(i,t)A(x)(i,t)\\
&=0.
\end{split}
\end{equation*}
\end{proof}

Now we describe how to determine the differential equations and their solutions. First, to define the unit normal vector, we need the condition for the unit normal vector at one vertex. The conditons for $a(0,t)$ or $b(0,t)$ provide this. Next, we describe the conditon for a solution to be determined. Since the equation $\frac{\partial}{\partial t}g_{ij}(n,t)=(r(t)-2K(n,t))g_{ij}(n,t)$ is defined for each layer, we have $2k$ equations. We note that the Gaussian curvature $K$ and the metric can be written in terms of
$$(f(0,t), \cdots ,f(k,t), h(1,t)-h(0,t), \cdots ,h(k,t)-h(k-1,t)).$$
When we consider the difference of $h$ as one function, the number of the above functions is $2k+1$. Therefore one more equation is required for the system to have determined solutions, and that equation is $f(k,t)=0$ in \eqref{eq:flow5} and \eqref{eq:flow6}. In \eqref{eq:flow7} and \eqref{eq:flow8}, $b(k,t)=1$ or $b(k,t)=-1$ are that condition. This is because $b(k,t)$ is determined by the differences of $f$ and $h$. The remaining condition $h(0,t)=0$ determines the values of $h(n,t)$. When we define
$$X(t):=(f(0,t), \cdots ,f(k,t), h(1,t)-h(0,t), \cdots ,h(k,t)-h(k-1,t)),$$
The four systems of equations in Section $6$ minus the condition $h(0,t)=0$ can be written as
\begin{equation}
\frac{d}{dt}X(t)=F(X(t)), \label{eq:flow9}
\end{equation}
where $F:\mathbb{R}^{2k+1} \rightarrow \mathbb{R}^{2k+1}$ is a particular mapping. As for \eqref{eq:flow5}, finding $F$ is not difficult:
\begin{multline*}
\frac{d}{dt}f(n,t)=\sum_{i=1}^{k-n-1}(-1)^{i-1}f(n+i,t)(K(n+i,t)-K(n+i-1,t))\\
+\frac{1}{2}(r(t)-2K(n,t))f(n,t)\;\;(n=0,\cdots,k-1),
\end{multline*}
$$\frac{d}{dt}f(k,t)=0,\;\;\frac{d}{dt}h(0,t)=0,$$
\begin{multline*}
\frac{d}{dt}(h(n+1,t)-h(n,t))=\sum_{i=1}^{k-n-1}2(-1)^{i}f(n+i,t)\frac{f(n+1,t)-f(n,t)}{h(n+1,t)-h(n,t)}\times \\
(K(n+i,t)-K(n-1+i,t))\cos^2\frac{\pi}{l}+\frac{1}{2}(r(t)-2K(n,t))(h(n+1,t)-h(n,t))\\
(n=0,\cdots,k-1).
\end{multline*}
It is possible to find $F$ for the three other systems of equations as well. Since \eqref{eq:flow9} is a system of $2k+1$ first order equations, there is a theorem for the existence and uniqueness of the solution of \eqref{eq:flow9} (see \cite{Arnold}).

\begin{fact}
A solution of the differential equation $\frac{d}{dt}X=F(t,X)$ with the initial condition $(t_0,X_0)$ in the domain of smoothness of the right-hand side exists and is unique.
\end{fact}

\begin{exa}
Finally, we consider the case that the initialized surface is a sphere, that is, 
$$x(m,n,0)=\Bigr\{\cos \frac{\pi n}{2k}\cos \frac{2\pi m}{l}, \cos \frac{\pi n}{2k}\sin \frac{2\pi m}{l}, \sin \frac{\pi n}{2k}\Bigr\}.$$
Then $f(n,0)=\cos \frac{\pi n}{2k}$ and $h(n,0)=\sin \frac{\pi n}{2k}$. When we set $a(0,0)=1$ and $b(0,0)=0$, this sphere has constant Gaussian curvature $K=1$. Moreover, since we have $f(k,0)=0$ and $a(k,0)=0$, we can consider both the flows \eqref{eq:flow5} and \eqref{eq:flow7}. We will find the solutions for \eqref{eq:flow5} and \eqref{eq:flow7} with initialized surface the above sphere are again that very sphere. When 
\begin{equation}
\left\{ \,
\begin{aligned}
&f(n,t)=\rho(t)f(n,0),\\
&h(n+1,t)-h(n,t)=\rho(t)(h(n+1,0)-h(n,0))\;\;(n=0, \cdots ,k-1),
\end{aligned}
\right. \label{eq:rho}
\end{equation}
$g_{11}(n,t)=\rho(t)^2g_{11}(n,0)$ and $g_{22}(n,t)=\rho(t)^2g_{22}(n,0)$ hold. Furthermore, then $a(n,t)=a(n,0)$, hence we have $K(n,t)=\frac{1}{\rho(t)^2}K(n,0)=\frac{1}{\rho(t)^2}$ and $r(t)=\frac{2}{\rho(t)^2}$. Substituting these into 
\begin{equation*}
\left\{ \,
\begin{aligned}
&\frac{\partial}{\partial t}g_{11}(n,t)=(r(t)-2K(n,t))g_{11}(n,t),\\
&\frac{\partial}{\partial t}g_{22}(n,t)=(r(t)-2K(n,t))g_{22}(n,t)\;\;(n=0, \cdots ,k-1),
\end{aligned}
\right.
\end{equation*}
we have $\rho(t)=1$. By uniqueness of the solution and the condition $h(0,t)=0$, the solution satisfies $f(n,t)=f(n,0)$ and $h(n,t)=h(n,0)$. 

We can also consider the unnormalized version for \eqref{eq:flow5} and \eqref{eq:flow7}, that is, the flow \eqref{eq:flow3.5} with conditions which are 
\begin{equation*}
a(0,t)=1,\;\;f(k,t)=0,\;\;h(0,t)=0
\end{equation*}
or
\begin{equation*}
a(0,t)=1,\;\;b(k,t)=1,\;\;h(0,t)=0.
\end{equation*}
In this case also, if we set $\rho(t)$ as in \eqref{eq:rho}, we have $\rho(t)=\sqrt{1-2t}$. 

These are normalized and unnormalized Ricci flow for discrete spheres.
\end{exa}

\end{document}